\documentclass[a4paper,twoside,12pt]{article}      
\usepackage{amsmath,amssymb,amsfonts} 
\usepackage{amsthm}
\usepackage{graphics}                 
\usepackage{color}                    
\usepackage{hyperref}                 
\usepackage{graphicx,eucal}
\usepackage{latexsym,epsfig,epic,eepic}

\newcommand{\smin}{\mbox{\rm($\star$)}}
\newcommand{\sLmin}{\mbox{\rm($\Lambda\star$)}}

\newenvironment{pfof}[1]{\par\noindent\textbf{Proof of #1.}}{\hfill $\square$}

\newcommand{\bbZ}{\mathbb{Z}}

\newcommand{\bbC}{\mathbb{C}}

\newcommand{\mJ}{\mathcal{J}}
\newcommand{\hmI}{\mathcal{I}}

\newcommand{\mQ}{\mathcal{Q}}

\newcommand{\bM}{\widetilde{M}}
\newcommand{\bS}{\widetilde{S}}

\newcommand{\tN}{\widetilde{N}}
\newcommand{\tR}{\widetilde{R}}
\newcommand{\tU}{\widetilde{U}}

\newcommand{\mC}{\mathcal{C}}
\newcommand{\eps}{\varepsilon}

\newcommand{\JR}{\Phi_R}
\newcommand{\JL}{\Phi_L}

\newcommand{\hI}{\hat{I}}

\newcommand{\const}{\mathrm{const}}

\newcommand{\NE}{\mathrm{NE}}
\newcommand{\Diff}{\mathrm{Diff}}
\newcommand{\id}{\mathrm{id}}

\newcommand{\wR}{R}

\newtheorem{theorem}{Theorem}[section]

\newtheorem{proposition}[theorem]{Proposition}
\newtheorem{corollary}[theorem]{Corollary}
\newtheorem{lemma}[theorem]{Lemma}
\newtheorem{question}[theorem]{Question}

\newtheorem{example}[theorem]{Example}

\theoremstyle{definition}
\newtheorem{definition}[theorem]{Definition}
\newtheorem{remark}[theorem]{Remark}

\newcommand{\Fix}{\mathop{\mathrm{Fix}}}

\newcommand{\Leb}{\mathrm{Leb}}
\newcommand{\PSL}{\mathrm{PSL}}
\newcommand{\Sc}{\mathbb{S}^1}
\newcommand{\bg}{\overline{\gamma}}
\newcommand{\Sf}{\widehat{S}}
\newcommand{\tC}{\widetilde{\mC}}

\newcommand{\emr}[1]{\textcolor{black}{#1}}  

\date{\today}
\title{On the ergodic theory of free group actions by real-analytic circle diffeomorphisms}

\author{Bertrand Deroin, Victor Kleptsyn, Andr\'es Navas}

\begin{document}
\maketitle
\begin{abstract}
We consider finitely generated groups of real-analytic circle diffeomorphisms. We show that if such a group admits 
an exceptional minimal set ({\em i.e.}, a minimal invariant Cantor set), 
then its Lebesgue measure 
is zero; moreover,  there are only finitely many orbits of connected components 
of its complement. 
For the case of minimal actions, we show that if 
the underlying group is (algebraically) free, then the action is ergodic with respect to the Lebesgue measure. 
This provides first answers to questions due to \'E. Ghys, G. Hector and D. Sullivan. 
\end{abstract}

\tableofcontents

\section{Introduction}

\subsection{Overview and statements of results}

This work is mainly motivated by the next longstanding questions in the theory of codimension-one 
foliations.  As far as we know, Question~\ref{one} and the first half of Question~\ref{two} below go 
back to D.~Sullivan and \'E.~Ghys, whereas the second half of Question \ref{two} goes back to 
G.~Hector.

\begin{question}\label{one}
Let $\mathcal{F}$ be a codimension-one foliation on a compact manifold~$M$ that is transversally 
of class $C^2$. Assume that $\mathcal{F}$ is minimal, that is, the 
closure of every leaf  is the whole manifold. Is it true that $\mathcal{F}$ is ergodic with respect 
to the transversal Lebesgue measure ? In other words, if $A$ is a measurable union of leaves, is it true 
that either $\Leb (A) = 0$ or $\Leb (M \setminus A) = 0$? 
\end{question}

\begin{question}\label{two}
Let $\mathcal{F}$ be a codimension-one foliation on a compact manifold that is transversally of class 
$C^2$. Assume that $\mathcal{F}$ admits an exceptional minimal set. Does this set have zero 
Lebesgue measure, and does its complement have only finitely many connected components?
\end{question}

Many interesting examples of foliations are obtained by the suspension of a group action (see \cite{CC}). 
Actually, this provides a natural framework for testing potential conjectures for general foliations. In the 
codimension-one context, the underlying space (fiber) should be obviously the circle, and the questions 
above translate into the next ones:

\begin{question}\label{q:m}
Let $G$ be a finitely generated group of $C^2$ circle diffeomorphisms. Assume that $G$ acts 
minimally, that is, every orbit is dense in the circle. Is the action necessarily ergodic with respect to the Lebesgue 
measure ? In other words, is it necessarily true that for every \emph{measurable} $G$-invariant subset 
$A\subset \Sc$, one has either $\Leb(A)=0$ or $\Leb (\Sc\setminus A) =0$?
\end{question}

\begin{question}\label{q:m-e}
Let $G$ be a finitely generated group of $C^2$ circle diffeomorphisms. 
Assume that the action of $G$ admits a minimal invariant Cantor set $\Lambda$ (also called 
an exceptional minimal set). Does $\Lambda$ necessarily have zero Lebesgue measure? Is the set 
of orbits of connected components of \hspace{0.03cm} $\mathbb{S}^1 \setminus \Lambda$ finite?
\end{question}

For simplicity, in what follows we will assume that all diffeomorphisms preserve the orientation. 
Notice that the general case reduces to this one, as the subgroup of orientation-preserving 
elements has index two in the original group, hence it carries all its relevant dynamical properties.

The aim of this work is to provide proofs of affirmative answers to Questions~\ref{q:m} 
and~\ref{q:m-e} for actions by real-analytic diffeomorphisms, with the extra hypothesis 
that the group $G$ is algebraically free for the first one. 
The reason for just treating 
free-group actions comes from that the very simple algebraic structure allows giving more 
elementary combinatorial arguments. Moreover, according to a theorem of Ghys to be 
discussed below, it still allows dealing with actions of non-necessarily free groups admitting 
an exceptional minimal set. We will discuss possible 
generalizations of these results in \S{}\ref{s:next} below; for the moment, we mention that 
our approach and techniques yield a road towards Question~\ref{two} and 
eventually to Question~\ref{one}. Nevertheless, more general cases will be treated
in separate works.

Finally, let us point out that translating our results and approach to the language of 
codimension-one foliations also seems to provide a road towards the well-known 
question of the topological invariance of the Godbillon--Vey class \cite{ghys-god}. 
Namely, given a codimension-one foliation, one can ask whether its holonomy 
pseudogroup is (locally) discrete. In the case it is, an analogue of our Main Theorem, 
combined with Theorem~\ref{t:KF} below would give us a rigid description 
of the dynamics that would probably suffice to describe the GV class. On the other hand, 
if the holonomy pseudogroup is not (locally) discrete, a topological conjugacy is often 
(transversally) analytic, which implies the invariance of the GV class almost immediately.

\vspace{0.18cm}
 
Let us begin by reviewing some of the literature on the subject. This will allow us to state our main 
(and somewhat technical) theorem from which the results announced above will directly follow.  

First, in what concerns minimal actions, if the group is generated by a single diffeomorphism, 
then its ergodicity is guaranteed by a theorem independently proven by A. Katok in the $C^{1+bv}$ 
case \cite{Katok}, and by M. Herman~\cite{Herman} in the $C^{1+Lip}$ case:

\begin{theorem}[Katok--Herman]
The action of every minimal $C^2$ circle diffeomorphism (equivalently, of every $C^2$ diffeomorphism 
with irrational rotation number) is ergodic with respect to the Lebesgue measure.
\end{theorem}

This theorem easily extends to (minimal) group actions with an invariant probability measure (such a group 
necessarily contains an element of irrational rotation number -and is conjugated to a group of rotations- 
provided it is finitely generated). For richer actions, a partial answer to Question~\ref{q:m-e} comes 
from Sullivan's exponential expansion strategy (see for instance~\cite{Navas-enseignement, Shub-Sullivan}):

\begin{theorem}[Sullivan]
Let $G$ be a group of $C^{1+\alpha}$ circle diffeomorphisms, with $\alpha > 0$. Assume that for all 
$x \in \Sc$ there exists $g\in G$ such that $g'(x)>1$. If the action of $G$ is minimal, then it is 
ergodic with respect to the Lebesgue measure.
\end{theorem}

A similar panorama arises when dealing with actions by $C^2$ diffeomorphisms with an exceptional 
minimal set (that is, a minimal invariant Cantor set) $\Lambda$. Indeed, in such a case, no invariant 
probability measure can exist provided the group is finitely generated (this is essentially a 
consequence of the Denjoy theorem). Moreover, if for every $x \in \Lambda$ there exists $g \in G$ 
such that $g'(x) > 1$, then the Lebesgue measure of $\Lambda$ can be easily shown to be zero. 

The approach above was pursued by S. Hurder~\cite{Hurder}, who defined a Lyapunov expansion exponent 
for the action: 
$$\mathcal{L}(x) := \limsup_{n \to \infty} \frac{1}{n} \max_{g \in B(n)} \log(g'(x)),$$
where $B(n)$ denotes the ball of radius $n$ (centered at $id$) 
with respect to some fixed finite system of generators of $G$.
He proved that if $G$ is made of $C^{1+\alpha}$ diffeomorphisms, $\alpha > 0$, 
then this exponent is constant Lebesgue almost everywhere on $\Sc$ for minimal 
actions, and constant Lebesgue almost everywhere on $\Lambda$ in case of 
an exceptional minimal set $\Lambda$.  
Moreover, the ergodicity in the former case and the zero Lebesgue measure of $\Lambda$ in the latter one 
remain true whenever the exponent is positive. The main problem with this approach is that 
in all the examples we know where the Lyapunov expansion exponent is positive (including those of 
$\mathrm{PSL}(2,\bbZ)$ and Thompson's group~$T$ to be recalled later), expanding 
elements appear everywhere ({\em resp.}, everywhere on~$\Lambda$). Actually, 
Corollary~\ref{c:zero-Lyap} below shows that, under certain hypothesis, positive Lyapunov 
expansion exponents imply the existence of expanding maps everywhere.  

It becomes clear from the previous discussion that a general obstacle for the 
application of the expansion strategy is the presence of non-expandable points:

\begin{definition}\label{d:ne}
A point $x\in\Sc$ is \emph{non-expandable} for the action of a subgroup 
$G \subset \mathrm{Diff}_+^1 (\mathbb{S}^1)$ if for all $g\in G$ one has $g'(x)\le 1.$ 
The set of non-expandable points is denoted by $\NE = \NE(G)$ (we omit $G$ in this notation as the group 
is usually fixed).
\end{definition}

The presence of non-expandable points is compatible with both minimality (with no invariant 
probability measure) and exceptional minimal sets for actions of finitely-generated groups. 
Such examples are known to exist in the real-analytic context, a relevant family being 
given by $\PSL(2,\bbZ)$ and other Fuchsian groups corresponding to surfaces with cusps 
(see \cite{sv-katok}). There are also the Ghys-Sergiescu smooth realizations of Thompson's 
group~$T$ (see~\cite{Ghys-Sergiescu}), yet these are only~$C^{\infty}$. We refer the reader 
to~\cite{DKN-MMJ} for a general discussion on these two families of actions. 
Nevertheless, it is worth mentioning that \emph{a posteriori}, such examples seem to be rare, 
forming a ``thin'' (though very interesting~!) ``boundary''  between two ``more stable'' types of actions,  
namely, those having an exceptional minimal set, and those that are minimal and have a  ``rich'' dynamics 
({\em e.g.} non locally discrete, that is, having local flows in their local closures; {\em c.f.} Proposition~\ref{p:affine}).

In an attempt to handle (and understand) the case where $\NE$ is nonempty and intersects the minimal set 
(this is quite rare in the minimal setting, but constructing examples with an exceptional minimal set is easier), 
the following notions were introduced in~\cite{DKN-MMJ}.

\begin{definition}\label{d:star}
A subgroup $G\subset \Diff^2_+(\Sc)$ satisfies \emph{property~\smin} if it is finitely generated, acts minimally, and for 
every $x\in\NE$ there exist $g_+,g_-$ in $G$ such that $g_+(x)=g_-(x)=x$ and $x$ is an isolated-from-the-right 
(\emph{resp.}, isolated-from-the-left) point of the set of fixed points $\Fix(g_+)$ (\emph{resp.}, $\Fix(g_-)$).
\end{definition}

\begin{definition}\label{d:lambda-star}
A subgroup $G \subset \Diff^2_+(\Sc)$ satisfies \emph{property~\sLmin} if $\Lambda$ is a Cantor set 
in $\Sc$ which is a minimal invariant set for the action of $G$ and for every $x\in\NE \cap \Lambda$ 
there exist $g_+,g_-$ in $G$ such that $g_+(x)=g_-(x)=x$ and $x$ is an isolated-from-the-right 
(\emph{resp.}, isolated-from-the-left) point of $\Fix(g_+)$ (\emph{resp.}, $\Fix(g_-)$).
\end{definition}

\begin{remark}\label{r:Dw}
Notice that for subgroups of the group $\Diff^{\omega}_+(\Sc)$ of real-analytic diffeomorphisms of the circle,  
the conditions about fixed points above are equivalent to that for all $x \in \NE$ ({\em resp.}, $x \in \NE \cap \Lambda$), 
there exists $g\in G\setminus \{\id\}$ such that $g(x)=x$. 
\end{remark}

Whenever property $\smin$ ({\em resp.}, $\sLmin$) holds, one can still show that the action is ergodic with respect 
to the Lebesgue measure ({\em resp.}, that $\Leb (\Lambda) = 0$), as well as many other interesting properties. All  
of this is summarized in the next three theorems below. Essentially, this follows from a combination of Sullivan's 
expansion strategy (performed away from non-expandable points) and classical parabolic expansion (performed 
close to these points).

\begin{theorem}[\cite{DKN-MMJ}]\label{t:DKN}
Let $G\subset \Diff^2_+(\Sc)$ be a group that satisfies property~\smin. Then:
\begin{enumerate}
\item The action of $G$ is ergodic with respect to the Lebesgue measure.
\item The set $\NE(G)$ is finite.
\item The set of points of bounded expansibility, that is, 
$$\big\{ x\in \Sc \mid \exists C: \quad \forall g\in G, \quad g'(x)\le C \big\},$$ 
coincides with the union $G(\NE)$ of orbits of points of $\NE$.
\end{enumerate}
\end{theorem}

The following theorem also has a version for $C^2$ diffeomorphisms, but for simplicity 
we only refer here to the $C^{\omega}$ version. 

\begin{theorem}[\cite{FK1,FK2}]\label{t:KF}
Let $G\subset \Diff^{\omega}_+(\Sc)$ be a group satisfying property~\smin{}   
and such that $\NE(G)\neq\emptyset$. Then:
\begin{enumerate}
\item\label{i:Markov} There exists a Markov partition for the action. More precisely, there is a finite partition 
of the circle $\Sc$ minus finitely many points into open intervals $I_j$, each provided with an element 
$g_j\in G$ and a set $\mathcal{I}_j$ of indexes so that 
$$
\quad g_j(I_j) = \bigcup_{k\in \mathcal{I}_j} I_k,
$$
so that for each $j$, at most one endpoint of $I_j$ is non-expandable. Moreover, 
for any $j$, one has $g_j'|_{I_j}>1$, and the strict inequality still holds 
at the endpoints of $I_j$ except for those which are non-expandable. 
Furthermore, if $x \in \partial I_j \cap \NE$, then $g_j (x) = x$ and $g_j'(x) = 1$. 
\item\label{i:orbits} A generic $G$-orbit is a union of finitely many full orbits of the 
action of the locally non-strictly-expanding map $R$ defined by $R|_{I_j}=g_j|_{I_j}$.
\item The Lyapunov expansion exponent of $G$ vanishes Lebesgue almost-everywhere.
\end{enumerate}
\end{theorem}

\begin{remark} 
The notion of a Markov partition used here is weaker than the one introduced by 
J. Cantwell and L. Conlon in~\cite{CC1, CC2}: whereas ours is adapted to the expansion 
procedure,  the one in~\cite{CC1, CC2} requires additionally that the full orbits of 
$R$ coincide with those of $G$. However, the obtained partition ``almost'' satisfies 
the definition in the Cantwell-Conlon sense. Indeed, conclusion~\ref{i:orbits}) above 
says that orbits of $G$ are decomposed into at most a finite number of $R$-orbits.
\end{remark}

\begin{theorem}[\cite{DKN-MMJ}]\label{t:DKN; version exceptional}
Let $G\subset \Diff^2_+(\Sc)$ be a group that satisfies property~\sLmin. Then:
\begin{enumerate}
\item The Lebesgue measure of $\Lambda$ is zero.
\item The set $\NE(G) \cap \Lambda$ is finite.
\item The set of points in $\Lambda$ of bounded expansibility, that is, 
$$\big\{x \in \Lambda \mid \exists C: \quad \forall g\in G, \quad g'(x)\le C \big\},$$ 
coincides with the union $G(\NE) \cap \Lambda$ of orbits of points of $\NE \cap \Lambda$.
\end{enumerate}
\end{theorem}

As far as we know, in the literature, for all minimal ({\em resp.}, with an exceptional minimal set $\Lambda$) and 
sufficiently smooth actions of  finitely-generated groups on the circle, property $\smin$  ({\em resp.}, \sLmin) 
does hold. The main result of this work states that, at least under certain assumptions, this is necessarily the case.

\vspace{0.3cm}

\noindent{\bf Main Theorem.}
{\em Let $G$ be a finitely-generated subgroup of $\Diff^{\omega}_+(\Sc)$.} 
\begin{enumerate}
\item {\em If G is free of rank $\geq 2$ and acts minimally on the circle, then it satisfies property~\smin.}
\item {\em If $G$ acts on the circle with an exceptional minimal set $\Lambda$, then it satisfies property~\sLmin.}
\end{enumerate}


Recall that a celebrated theorem of Ma\~n\'e (see~\cite[Thm.~A]{Mane}) says that 
for a $C^2$ circle {\em endomorphism} that is not conjugated to an irrational rotation, 
every closed invariant set having no critical point and which is not hyperbolic must contain 
a parabolic fixed point. 
In a certain sense, our Main Theorem is an analogous of this result, though in our 
case we show that \emph{every} point of the minimal set that cannot 
be expanded is a parabolic fixed point of one of the elements of the group.

\begin{corollary}\label{c:erg}
If $G\subset \Diff^{\omega}_+(\Sc)$ is a finitely-generated free group acting 
minimally, then its action is ergodic with respect to the Lebesgue measure.
\end{corollary}

Indeed, for rank 1, this is Katok-Herman's theorem, whereas for higher rank, 
this follows from the Main Theorem.1) together with Theorem~\ref{t:DKN}.1).

\begin{corollary}\label{c:zero-Lyap}
Let $G\subset \Diff^{\omega}_+(\Sc)$ be a finitely-generated group that is free of rank at least~$2$. 
Assume that~$G$ acts minimally and that~$\NE(G)$ is nonempty. Then the Lyapunov expansion 
exponent vanishes Lebesgue almost everywhere.
\end{corollary}

Again, this follows from the Main Theorem.1) together with Theorem \ref{t:KF}.3).

\begin{corollary}\label{c:Lm}
If $G\subset \Diff^{\omega}_+(\Sc)$ is a finitely-generated group preserving an exceptional 
minimal set $\Lambda$, then $Leb(\Lambda) = 0$. 
\end{corollary}

This follows from the Main Theorem.2) together with Theorem \ref{t:DKN; version exceptional}.1). 

\begin{corollary}\label{c:finite-comp}
If $G\subset \Diff^{\omega}_+(\Sc)$ is a finitely-generated group preserving an exceptional 
minimal set $\Lambda$, then $\mathbb{S}^1 \setminus \Lambda$ is the union of finitely 
many orbits of intervals.
\end{corollary}

This corollary corresponds to an analog of the famous Ahlfors' finiteness theorem~\cite{ahlfors},   
and as we already mentioned, it answers in the affirmative a question of G. Hector. Its proof 
requires more discussion, hence it is postponed to~\S{}\ref{s:modifications}, where we also 
further develop on the proof of the preceding corollary.


\subsection{Sketch of the proof and plan of the article}

Throughout the proof, we will constantly use classical tools of control of distortion which 
we briefly recall in \S~\ref{s:prelim-1}. In most cases, we will use them without details 
(these are left to the reader), except whenever some precise and explicit estimate is needed. 

In their original formulation, control of distortion techniques go back to the works of A. Denjoy 
\cite{denjoy}, R. Schwartz \cite{Sc} and R. Sacksteder \cite{Sa}. These consist of tools that 
allow to compare a composition of ``simple'' $C^2$-diffeomorphisms, restricted to some 
``small'' interval, to the corresponding affine map, provided that the sum of the lengths of 
the intermediate images of this interval is not too big. A nice view of this was later proposed 
by Sullivan \cite{Sullivan}, who cleverly noticed that the ``right hypothesis'' is the existence of 
an upper bound for the sum of intermediate derivatives at \emph{a single} point~$y$. In concrete terms, 
knowing that this sum does not exceed some constant $S$, then in a neighborhood of radius $\sim 1/S$ 
of $y$, the composition looks like an affine map (despite the number of the involved compositions could 
be very large !). We will also use certain results on composition of one-dimensional holomorphic maps, 
more specifically the commutators arguments and the vector fields technique. Both 
are recalled in \S~\ref{s:prelim-2}.

We then proceed to the proof of the Main Theorem. We will mainly focus on item~1), as 2) will follow with some 
minor adjustments to be commented in \S~\ref{s:modifications}. 
\emr{We assume that there is a non-expandable point $x_0 \in \Sc$. Our goal is to show that it is fixed by some nontrivial 
element of~$G$. To do this, let~$\mathcal{G}$ be the canonical systems of generators of $G$ (together with their inverses), 
and let $B(n)$ be the ball of radius $n$ centered at the identity:
$$
B(n):=\{g_k\circ \cdots \circ g_1 \mid \forall j \!:  g_j\in \mathcal{G}, \hspace{0.1cm} k\le n\}.
$$
Denote the sum of the derivatives of the elements of this ball at $x_0$ by $S_n$:
$$
S_n := \sum_{g\in B(n)} g'(x_0).
$$}\noindent \emr{The proof of the Main Theorem will then be obtained from the following two 
statements:} 

\vspace{0.1cm}

\noindent {\bf (A)} \emr{If the sums $S_n$ grow faster than linearly, then the point $x_0$ is stabilized by a nontrivial element.  
(This is  Proposition~\ref{p:linear} in~\S{}\ref{s:proof}.)} 

\vspace{0.1cm}

\noindent {\bf (B)} 
\emr{If they grow slower than exponentially, then one can construct a Markov partition ``by hands'' 
and directly deduce property $(\star)$ from it.  (This is Theorem~\ref{t:exp}  in~\S{}\ref{s:growth}.)} 

\vspace{0.1cm}

\noindent \emr{Clearly, these two statements together prove the Main Theorem.}

\vspace{0.1cm}

 
\emr{Let us sketch the proofs of these statements. To prove Proposition~\ref{p:linear}, we proceed by contradiction.  
We assume that $x_0$ is not fixed by any nontrivial $g\in G$, and we consider all the images of $x_0$ by the elements 
of $B(n)$. Let $x_n := f_n(x) \neq x_0$ be the closest point to it among those of the form $f(x_0)$, where $f \in B(n)$. 
We first deduce from the superlinear growth assumption that the length of the 
interval~$I_n:=[x_0,x_n)$ decreases as $o(1/n)$. Indeed, the images $g(I_n)$ as $g$ ranges over $B([n/2])$  
are pairwise disjoint (see Lemma~\ref{l:intersection}). A control of distortion argument then 
provides us with the desired estimate $|I_n|\le \frac{\const}{S_{[n/2]}} = o(1/n)$; � see Lemma~\ref{l:conv-int}.
} 

\emr{Now, another control of distortion argument together with the non-expandability of the point $x_0$ allow us to conclude 
that the maps $f_n$  look closer and closer to the identity in some $o(1/n)$-neighborhood of $x_0$ (see 
Lemma~\ref{l:id}). We then find two consecutive non-commuting maps $f_n$, $f_{n+1}$ (see 
Lemma~\ref{l:free-gen}), and apply Ghys' commutator technique to obtain a sequence of commutators that  
converges to the identity on some interval. Finally, the Shcherbakov-Nakai-Loray-Rebelo technique applied
to this sequence allows us to find local flows in the local closure of the group~$G$ (see Proposition~\ref{p:affine}), 
and this contradicts the presence of non-expandable points (compare Remark~\ref{r:NE-C1}). 
}

\emr{The proof of Theorem~\ref{t:exp} is more technical. A standard argument implies 
that for each point $x\in\Sc$ there is a geodesic path $\gamma$ in $G$, starting with the identity, 
along which the sum of the derivatives of the maps at~$x$ is arbitrarily large (see Proposition~\ref{p:geodesic-free}).
Assume now that the above property 
can be strengthened as follows: not only at any point~$x$ we can find geodesic paths with large sums of intermediate 
derivatives, but also we can find such geodesic paths starting with any prescribed generator. 
The first step in the proof of Theorem~\ref{t:exp} is 
to deduce from this assumption the desired exponential 
growth for the sums~$S_n$. To do this, we use a ``growing trees'' argument (see~\S{}\ref{s:growth}).
Namely, we inductively construct ``tree-like'' sets $\Gamma_m$ of  
\hspace{0.1cm} diameter $\le \const \cdot m$ \hspace{0.1cm} with 
an exponentially growing sum of the associated derivatives. 
This is done by replacing at each step the points of a previously 
constructed set by finite geodesic paths going into ``free cones'' at these points.}

\emr{Finally, the hardest part of the proof (which involves several combinatorial arguments 
that are particular to the free group) consists in showing that the failure of the strengthened 
property above implies property~$(\star)$. (This is  Proposition~\ref{p:geod}, 
proved in~\S{}\ref{s:cones}.) To do this, the main idea is the following: if the sums of the derivatives 
at some point along the geodesics  starting at some generator are uniformly bounded, then we have a 
uniform control of distortion for the maps in this cone on some neighborhood of this point. We then consider 
the maximal domains where such control holds (to be more precise, where boundedness of sums of derivatives holds), and 
show (see Lemmas~\ref{l:images},~\ref{l:disj},~\ref{l:lr-r},~\ref{l:finite}) that they form a Markov partition for the action of~$G$. 
Finally, the presence of a Markov partition allows us to obtain the property~$(\star)$ ``by hands''; in particular, the fact that Markovian 
maps become expanding at the interior of the intervals of the partition is established by Corollary~\ref{c:needed}.}

More technical issues and outcomes of the proof (see {\em e.g.} Remark \ref{r:rare}) will be 
considered along the text.

\subsection{What's next?}\label{s:next}

The proof of our Main Theorem uses both the facts that the acting group is free and that the action is by 
real-analytic diffeomorphisms. Also, for the exceptional minimal case, we use Ghys' theorem stating that a finitely 
generated group of real-analytic circle diffeomorphisms having an exceptional minimal set is virtually free. Thus, 
to apply our approach to the case of codimension-one foliations (or to the pseudo-group context), additional 
modifications would be required.

Despite this, a careful study of the instances when the assumptions (freeness, analyticity, group and not a 
pseudo-group) are used shows that most of these situations can be passed assuming weaker assumptions (though 
with more technical difficulties). This section is devoted to the possible generalizations: we list the instances 
where our assumptions have been used, and the possible lines of improvement.
\begin{itemize}
\item \emph{Analyticity} of the group action is used in Ghys' commutator argument. Though, convergence to the identity of a chain of commutators can be established also via (more technically complicated) $C^2$ (and even $C^{1+\eps}$) arguments (this was independently noticed by J. Rebelo and A. Eskif~\cite{ER} and by the authors).
\item The \emph{(virtual) freeness} of the group is used to establish the lower bound~\eqref{eq:to-prove} for the sums of the derivatives, as well as to ensure that in a certain chain of commutators, the maps do not eventually become the identity. For the first part ({\em c.f.} Proposition \ref{p:geod}), it seems that at least the statement itself can be translated into the language of ends of the acting group or of ends of the orbit of the pseudo-group. Indeed, the cones defined in \S~\ref{s:cones} are the (infinite) connected components of the complement of a ball (consisting in this case of a single point). Some generalizations in this direction have been already made in~\cite{A-etal}. On the other hand, the case of finitely-presented groups with one end has been studied in~\cite{FK-spheres}.
\item In the exceptional minimal case, we use both assumptions: we deal with a \emph{group} (and not only with a pseudogroup or a foliation) and the action is by real-analytic diffeomorphisms (this allows us applying Ghys' virtual freeness theorem). However, in the general $C^2$ pseudogroup context, we still have Duminy's theorem concerning infinitely many ends for semiproper leafs, which allows the arguments cited earlier have a good chance to work.
\item Finally, the \emph{analyticity} assumption is used two more times. One is to ensure that a fixed point of a non-identity map is an isolated fixed point (while obtaining property~$(\star)$). Here, it seems quite plausible that this assumption can be weakened (with the help of some dynamical arguments). The other one relies on Loray-Rebelo-Nakai-Shcherbakov's argument  on local flows in the closure of locally non-discrete (pseudo)groups ({\em c.f.} Proposition \ref{p:affine}). Again, it seems that the assumption here can be weakened. In fact, the main case of this argument (namely, that of an hyperbolic fixed point that is not fixed by maps converging to the identity) works even in $C^2$-regularity. The case left seems to be quite restricting, so it is highly probable that one can generalize the arguments also for this case.
\end{itemize}


\section{Preliminaries}\label{s:prelim}

\subsection{Control of distortion estimates}
\label{s:prelim-1}

We begin by recalling several lemmas concerning control of distortion which are classical in the context 
of smooth one-dimensional dynamics. A more detailed discussion with (references to the) proofs may 
be found in~\cite{DKN-MMJ}.

\vspace{0.1cm}

\begin{definition} Given two intervals $I,J$ and a $C^1$ map $g \!: I\to J$ which is a
diffeomorphism onto its image, we define the \emph{distortion coefficient} of $g$ on~$I$ by
$$
\varkappa (g;I) := \sup_{x,y \in I} \left| \log \Big( \frac{g'(x)}{g'(y)} \Big) \right|.
$$
\end{definition}

\vspace{0.1cm}

The distortion coefficient is subadditive under composition, that is 
$$\varkappa(fg,I) \leq \varkappa (g,I) + \varkappa (f, g(I)).$$ 
Moreover,  it satisfies 
\begin{equation}\label{eq:inverses}
\varkappa (g,I) = \varkappa (g^{-1},g(I)),
\end{equation}
as well as $\varkappa(g,I)\le C_{\{g\}} |I|,$ where the constant $C_{\{g\}}$ depends only on the
$\Diff^2$-norm of~$g$ (indeed, one can take $C_{\{g\}}$ as being the maximum of the absolute
value of the derivative of the function $\log(g')$). This immediately implies the following

\vspace{0.1cm}

\begin{proposition}\label{p:sum}
Let $\mathcal{G}$ be a subset of $\Diff^2_+(\Sc)$ that is bounded with respect to the
$\Diff^2$-norm. If $I$ is an interval of the circle and $g_1,\dots,g_n$ are finitely
many elements chosen from $\mathcal{G}$, then
$$
\varkappa(g_n \circ \cdots \circ g_1;I) \le
C_{\mathcal{G}} \sum_{i=0}^{n-1} |g_i \circ \cdots \circ g_1 (I)|,
$$
where the constant $C_{\mathcal{G}}$ depends only on~$\mathcal{G}$. 
(Here, $g_i \circ \cdots \circ g_1$ is the identity for $i = 0$.)
\end{proposition}

\vspace{0.1cm}

An almost direct consequence of this proposition is the next

\vspace{0.1cm}

\begin{corollary}\label{cor:estimates}
Under the assumptions of Proposition~\ref{p:sum}, let us fix a point $x_0 \in I$,
and let us denote $f_i := g_i \circ \cdots \circ g_1$, $I_i := f_i(I)$, and
$x_i := f_i(x_0).$ Then the following inequalities hold:
\begin{equation}
\exp \Big(\!-C_{\mathcal{G}}\sum_{j=0}^{i-1} |I_j| \Big) \cdot \frac{|I_i|}{|I|} \le f_i'(x_0)
\le \exp \Big( C_{\mathcal{G}}\sum_{j=0}^{i-1} |I_j| \Big) \cdot \frac{|I_i|}{|I|} ,
\end{equation}
\begin{equation}\label{eq:lsum}
\sum_{i=0}^n |I_i| \le  |I|
\exp \Big( C_{\mathcal{G}} \sum_{i=0}^{n-1} |I_i| \Big) \sum_{i=0}^{n} f_i'(x_0).
\end{equation}
\end{corollary}

\vspace{0.1cm}

Using this corollary, an inductive argument allows showing the following important 

\vspace{0.1cm}

\begin{proposition}\label{bound}
Under the assumptions of Proposition~\ref{p:sum}, given a point
$x_0\in \Sc$, let us denote \, $S := \sum_{i=0}^{n-1}
f_i'(x_0)$. \, Then for every \, $\delta \leq \log(2) / 2 C_{\mathcal{G}} S $, \, 
one has \,
$$\varkappa(f_n,U_{\delta/2}(x_0)) \le 2 C_{\mathcal{G}} S \delta,$$
where $U_{\delta/2}(x_0)$ denotes the $\delta/2$-neighborhood of $x_0$.
\end{proposition}

\begin{proof} 
By Proposition~\ref{p:sum}, it suffices to check that
$$
\sum_{i=0}^{n-1} |f_i(I)| \le 2\delta S,
$$
where $I := U_{\delta/2} (x_0)$. To do this, we proceed by induction on~$n$.
As $f_0$ is the identity and $|I|=\delta$, for 
$n = 1$ we have $S = 1$ and 
$$
|I|=\delta\le 2\delta S,
$$
hence the claim holds. Assume it holds for some $n$, and let us check it for $n+1$. 
Then by \eqref{eq:lsum} we have
\begin{multline*}
\sum_{i=0}^n |f_i(I)| \le \delta \cdot \exp \left( C_{\mathcal{G}}\cdot  \sum_{i=0}^{n-1} |f_i(I)| \right) \cdot \sum_{i=0}^n f_i'(x_0) \le \\ \le  \exp (2C_{\mathcal{G}} S \delta) \cdot \delta S \le \exp(\log 2) \cdot \delta S = 2\delta S,
\end{multline*}
where the last inequality is a consequence of the choice of~$\delta$. This closes the proof.
\end{proof}

\vspace{0.1cm}

As an application of the previous discussion, \emr{we show}

\begin{proposition}\label{p:geodesic-free}
Let $G$ be a finitely-generated free group of $C^2$ circle diffeomorphisms acting minimally, and let $\mathcal{G}$ be the set of its (standard) generators together with their inverses. Then for every point~$x$ of the circle, one can find elements (finite geodesics)  
$g = \gamma_n\cdots\gamma_1$, with $\gamma_i\in \mathcal{G}$ and $\gamma_i\neq \gamma_{i+1}^{-1}$ 
for each $i < n$, with arbitrarily large sum of intermediate derivatives at~$x$.
\end{proposition}

\begin{proof} Otherwise, due to Proposition~\ref{bound}, for a sufficiently small $\delta > 0$, the distortion on the 
neighborhood $U_{\delta/2}(x)$ of all elements in the group would be uniformly bounded, say by a constant $C > 0$. 
Now since the action of $G$ cannot preserve a probability measure (otherwise the group would be conjugate to a 
group of rotations, hence Abelian; see \cite[Lemma 4.1.8]{Navas-Es}), the extension of Sacksteder's 
theorem of \cite{DKN} yields an element $f \in G$ with an hyperbolically repelling fixed point $y_0 \in U_{\delta/2}$. 
Up to changing $f$ by some iterate if necessary, we may assume that $f'(y_0) > \frac{1}{\delta e^{C}}$. The upper 
bound for the distortion then yields $f'(y) > \frac{1}{\delta}$ for all $y \in U_{\delta/2}$. However, this is 
impossible, as it would imply that the image of $U_{\delta/2}$ under $f$ is larger than the whole circle.
\end{proof}

We will also need a ``complex version'' of Proposition~\ref{bound} (with the obvious extensions 
of definitions), the proof of which is analogous to that of the classical one and is left to the reader. 

\vspace{0.1cm}

\begin{proposition}\label{bound-complex}
Suppose that $\mathcal{G}$ is a finite subset of $\mathrm{Diff}_+^{\omega} (\mathbb{S}^1)$. 
Then there exists a constant $\rho > 0$ depending only on $\mathcal{G}$ such that the statement 
of Proposition \ref{bound} holds provided we add the condition $\delta \leq \rho$. More precisely, 
for a certain  constant $C_{\mathcal{G}} > 0$ and any point $x_0\in \Sc$, if we denote 
$$S := \sum_{i=0}^{n-1} f_i'(x_0),$$ 
where $f_k := g_k \circ \cdots \circ g_1$, $g_i \in \mathcal{G}$, then for every \, 
$\delta \leq \min \{ \log(2) / 2 C_{\mathcal{G}} S , \rho \}$, \, one has \,
$$\varkappa \big( f_n,U_{\delta/2}^{\mathbb{C}}(x_0) \big) \le 2 C_{\mathcal{G}} S \delta,$$
where $U_{\delta/2}^{\mathbb{C}}(x_0)$ denotes the complex $\delta/2$-neighborhood of $x_0$.
\end{proposition}


\subsection{Commutators and the vector fields technique}
\label{s:prelim-2}

The next two results will be crucial to deal with maps that behave like translations on some intervals.

\begin{proposition}[Ghys~\mbox{\cite[Prop.~2.7]{Ghys}}]\label{p:to-id}
There exists $\varepsilon_0>0$ with the following property. Assume that the analytic local diffeomorphisms  
$f_1,f_2:U_1^{\bbC}(0)\to \bbC$ are $\varepsilon_0$-close (in the $C^0$ topology) to the identity, 
and let the sequence $f_k$ be defined by the recurrence relation 
$$
f_{k+2}=[f_k,f_{k+1}], \quad k=1,2,3,\dots 
$$
Then all the maps $f_k$ are defined on the disc $U_{1/2}^{\bbC}(0)$ of radius~$1/2$, 
and $f_k$ converges to the identity in the $C^1$ topology on $U_{1/2}^{\bbC}(0)$.
\end{proposition}

The following proposition is in the spirit of results by A. Shcherbakov \cite{EISV}, I. Nakai~\cite{Nakai}, 
J. Rebelo~\cite{Rebelo3,Rebelo2}, and F. Loray and J. Rebelo~\cite{Loray-Rebelo}. However, though it 
seems to be well-known to specialists, the statement doesn't appear in the form below in the literature. 
For the reader's convenience, we provide a proof plus a short discussion. 

\vspace{0.1cm}

\begin{proposition}
\label{p:affine}
Let~$I$ be an interval on which certain real-analytic nontrivial diffeomorphisms $f_k\in G$ are defined. 
Suppose that $f_k$ converges to the identity in the $C^{\omega}$ topology on $I$, and let $f$ be 
another $C^{\omega}$ 
diffeomorphism having an hyperbolic fixed point on $I$. Then there exists a (local) $C^1$ 
change of coordinates $\varphi:I_0 \to [-1,2]$ on some subinterval $I_0\subset I$ after which the 
pseudo-group $G$ generated by the $f_k$'s and $f$ contains in its $C^1([0,1],[-1,2])$-closure 
a (local) translation subgroup:
$$\overline{ \{\varphi \circ g \circ \varphi^{-1}|_{[0,1]} \mid g\in G \} } 
\supset \{ x\mapsto x+s \mid s\in [-1,1]\}.$$
\end{proposition}

\begin{proof} By the Poincar\'e linearization theorem, we may assume that 
$f$ is affine on a  neighborhood of the hyperbolic fixed point $p$. If $f_k (p) \neq p$ 
for infinitely many $k$, then the claim follows from \cite[Proposition 3.1]{Rebelo3}. 
Assume $f_k (p) = p$ for all but finitely many~$k$. If $f$ does not commute with 
infinitely many $f_k$, then the claim follows from \cite[Section 3]{Nakai}. Otherwise, 
for all but finitely many $k$ we have that $f_k$ is affine around $p$, with multiplier 
converging to 1. The corresponding affine flow is hence contained in the closure of 
$G$ when restricted to the neighborhood; taking logarithmic coordinates, this 
becomes the desired translation flow.
\end{proof}

\begin{remark}\label{r:C1}
In the proposition above, one may relax the convergence of $f_k$ to the identity to hold 
only in class $C^1$. We sketch the proof for this case since it will have some relevance 
further on (see Remark~\ref{r:NE-C1}).  
\end{remark}

\begin{proof}[Sketch of the proof] 
Again, we will assume that $f$ is affine on a $\varepsilon$-neighborhood of the hyperbolic fixed 
point $p = 0$, say $f(x) = \lambda x$ for $x \in [-\varepsilon,\varepsilon]$, with $\lambda > 1$.  

If $f_k (0) = 0$ holds for infinitely many $k$, then we can still apply the above arguments 
(namely, Nakai's result \cite{Nakai} in case of noncommuting $f,f_k$, and the affine flow 
argument in case of commuting elements). Otherwise, we may follow an argument 
of~\cite{Rebelo3}. Namely,  fix a positive 
$\delta < \varepsilon / \lambda$, 
and let $g_k := f^{n_k} f_k f^{-n_k}$, where $n_k$ is to be defined. Since $f$ 
is linear on $[-\varepsilon,\varepsilon]$, for large-enough $k$ we have 
$$\sup \{ |g_k' (x) - 1| \!: x \in [-\varepsilon,\varepsilon] \} 
= \sup \{ |f_k' - 1| \!: x \in  [-\varepsilon/\lambda^{n_k},\varepsilon/\lambda^{n_k}] \}.$$
Now, for the choice $n_k := \left[ \log_{\lambda} \left( \frac{\delta}{|f_k (0)|} \right) \right]$, 
for all $k$ we have 
$\frac{\delta}{\lambda} < \big| g_k(0) \big| \leq \delta.$ 
Hence, for a certain subsequence of $g_k$, we get the $C^1$ convergence to a translation 
(by a certain $\pm t$, where $t \in [\delta/\lambda,\delta]$). 
As $\delta$ can be chosen arbitrarily small, the $C^1$ local closure of the group contains 
arbitrarily small translations.
\end{proof}


\section{Proof of the theorem}

\subsection{Exponential growth estimates}\label{s:growth}

In all what follows, unless otherwise explicitly stated, we assume 
that $G \subset \mathrm{Diff}_+^{\omega}(\mathbb{S}^1)$ is a free group in finitely many 
(though at least two) generators. Let us denote by $\mathcal{G}$ its standard generating system 
(in which we include the inverses of all the generators). For simplicity, whenever 
we write an element $g \in G$ in the form $g = \gamma_n \cdots \gamma_1$, we will implicitly assume that this is 
its reduced expression, that is, each $\gamma_i$ belongs to $\mathcal{G}$ and $\gamma_{i+1} \neq \gamma_i^{-1}$. 
We will also think of these expressions as {\em reduced words} or {\em geodesics}. Since generators are composed 
(multiplied) from right to left, we will call a {\em suffix} of $g$ an expression of the form $\gamma_n \cdots 
\gamma_k$, $1 \leq k \leq n$, while an expression like $\gamma_{k} \cdots \gamma_1$ will be a {\em prefix} of $g$.

Given $g = \gamma_n \cdots \gamma_1$ in $G$, 
we will call the {\em cone based at $g$} the set $\mC_g$ of elements of the 
form $\overline{g} g$, where $\overline{g} = \overline{\gamma}_{m} \cdots \overline{\gamma}_1$ 
satisfies $\overline{\gamma}_1 \neq \gamma_n^{-1}$. (Notice that $g$ does not belong to $\mC_g$.)

This paragraph is devoted to the proof of the exponential growth for the sum of the derivatives:

\begin{theorem}\label{t:exp}
Let $G \subset \Diff^{\omega}_+(\Sc)$ be a finitely-generated free group acting minimally. 
Then $G$ satisfies property \smin, or there exist positive constants $c, \lambda$ 
such that for all $x\in\Sc$ and all $n \geq 1$, 
\begin{equation}\label{exp-sums}
\sum_{g\in B(n)} g'(x)\ge c \hspace{0.03cm}e^{\lambda n}.
\end{equation}
\end{theorem}

We will deduce this result from the following proposition, the proof of which we postpone untill \S{}\ref{s:cones}:

\begin{proposition}\label{p:geod}
Let $G \subset \Diff^{\omega}_+(\Sc)$ be a finitely-generated free group acting minimally. 
Then $G$ satisfies property~\smin, or for all $x\in\Sc$ and 
all $\gamma \in \mathcal{G}$, there exists $g=\gamma_k\cdots \gamma_1$ in $\mC_{\gamma}$ satisfying
$$
\sum_{j=1}^k (\gamma_j\cdots \gamma_1)'(x) > 2.
$$
\end{proposition}

Notice that the conclusion of this proposition is not far away from that of Proposition~\ref{p:geodesic-free}. Indeed, from 
that proposition we know that for each point $x$, there are geodesics with arbitrarily big sums of derivatives along the 
compositions. Nevertheless, here we stress this property (in absence of property \smin) by asking for such geodesics 
in each of the cones $\mC_{\gamma}, \, \gamma \in \mathcal{G}$. 

Let us now deduce Theorem~\ref{t:exp} from Proposition \ref{p:geod}:

\vspace{0.3cm}

\begin{pfof}{Theorem~\ref{t:exp}}
Assume that $G$ satisfies the assumptions of Theorem~\ref{t:exp} but does not satisfy the property~\smin. 
Then, due to Proposition~\ref{p:geod}, for every point $x\in\Sc$ and for every cone $\mC_{\gamma}, \, 
\gamma \in \mathcal{G}$, we can find a geodesic $g=\gamma_k\cdots \gamma_1\in \mC_{\gamma}$ 
such that the sum of the intermediate derivatives along $\gamma$ exceeds~$2$:
\begin{equation}\label{eq:sg}
\sum_{j=1}^k (\gamma_j\cdots \gamma_1)'(x) > 2.
\end{equation}
Obviously, this inequality still holds in a small neighborhood of the initial point $x$. Hence, by the 
compactness of the circle, we can choose a finite set $\mathcal{F}$ of possible elements $g$. In 
particular, we can assume that the length of each of these $g$ does not exceed some constant~$L$. 

We claim that the following exponential lower bound holds for each $x \in \mathbb{S}^1$ and all~$n$:
\begin{equation}\label{eq:to-prove}
\sum_{g \in B(n)} g'(x) \ge 2^{[n/L]}.
\end{equation}
To prove this, we will actually prove a stronger statement. Namely, let $x \in \mathbb{S}^1$ be fixed. 
We will show that for every $m \geq 0$, there exist a subset $\Gamma_m \subset B(mL)$ and a map 
$\gamma \!: \Gamma_m \to \mathcal{G}$ 
(both depending on $x$~!), so that the 
associated cones $\mC_{\gamma(g)}g$, with $g \in \Gamma_m$, start at $g$ and satisfy:
\begin{enumerate}
\item For each $g \in \Gamma_m$, the cone $\mC_{\gamma(g)} g$ 
contains no point of $\Gamma_m$ (hence these cones are all mutually disjoint).
\item 
One has $\sum_{g \in \Gamma_m} g'(x) \ge 2^{m}$.
\end{enumerate}

The proof proceeds by induction on~$m$. For $m=0$, one can take $\Gamma_0 := \{\id\}$ and any function 
$\gamma : \Gamma_0 \to \mathcal{G}$. 
Assume that for some $m$, the set $\Gamma_m$ and the map $\gamma$ have been constructed, and let us construct 
them for $\widetilde{m}:=m+1$. To do this, for each $g \in \Gamma_m$,  
consider the point $y := g(x)$, and take 
$\overline{g}=\overline{\gamma}_k\cdots\overline{\gamma}_1$ in $\mC_{\gamma(g)} \cap \mathcal{F}$ 
(with~$k \leq L$) 
so that inequality 
(\ref{eq:sg}) holds for $\overline{g}$ at the point $y$. Next, take $\Gamma_{m+1}$ to be the set of elements 
of the form $\overline{\gamma}_j\cdots\overline{\gamma}_1 g$, where $g$ runs over all possible elements 
in $\Gamma_m$ and $1\le j \le k$, with $\overline{g}$ associated to $g$ 
as above. We then define the new map $\widetilde{\gamma}$ on $\Gamma_{m+1}$ by 
letting $\widetilde{\gamma}(\overline{\gamma}_j\cdots \overline{\gamma}_1 g)$ to be equal to any element 
of $\mathcal{G}$ different from $(\overline{\gamma}_j)^{-1}$ and (in case $j<k$) from $\overline{\gamma}_{j+1}$. 
(See Fig.~\ref{f:Gamma}.)

\begin{figure}
\includegraphics[scale=0.7]{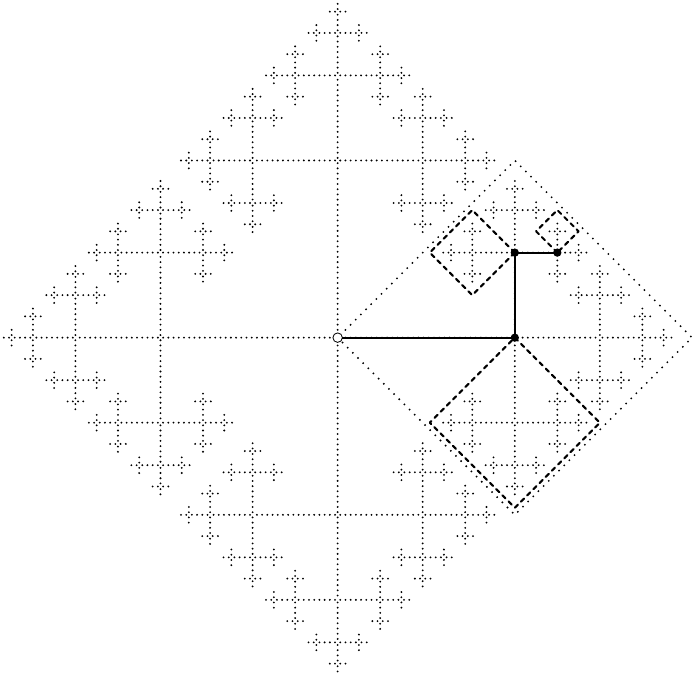}  \qquad
\includegraphics[scale=0.52]{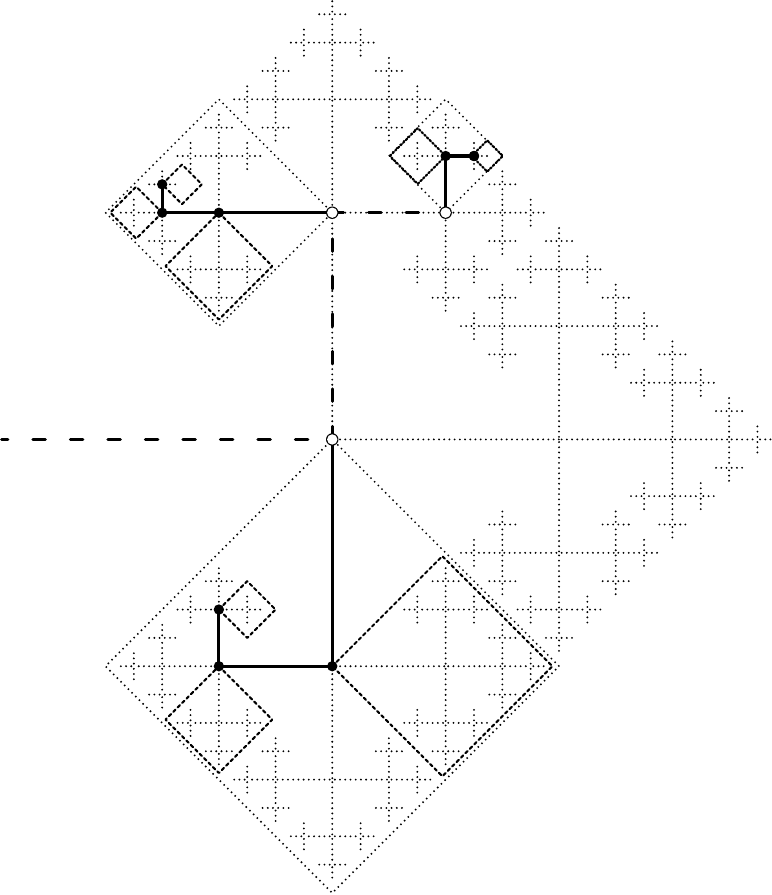}
\caption{On the left, the first step of the construction: 
$\Gamma_0=\{\id\}$, $\Gamma_1$ is shown by bold circles. 
On the right, the second step: $\Gamma_1$ is shown by empty circles, 
$\Gamma_2$ by bold ones.}\label{f:Gamma}
\end{figure}

By construction, $\Gamma_{m+1} \subset B_{L(m+1)}(e)$. Disjointness of cones also follows from the 
construction. Moreover, we have 
\begin{multline*}
\sum_{\overline{g} \in \Gamma_{m+1}} \overline{g} '(x) 
= 
\sum_{g \in \Gamma_m} \sum_{j=1}^k (\overline\gamma_j \cdots \overline\gamma_1 g)'(x) 
=  \sum_{g \in \Gamma_m} \sum_{j=1}^k (\overline\gamma_j \cdots \overline\gamma_1)' (g (x)) \cdot  g'(x) 
= \\
= \sum_{g \in \Gamma_m} g'(x) \sum_{j=1}^k (\overline\gamma_j \cdots \overline\gamma_1)' (g(x)) 
\ge 2\sum_{g \in \Gamma_m}  g'(x) \ge 2 \cdot 2^{m} = 2^{\widetilde{m}}.
\end{multline*}
This concludes the inductive proof. 

Finally,  
$$\sum_{g \in B(n)} g'(x) \ge \sum_{g \in \Gamma_{[n/L]}} g '(x) \ge 2^{[n/L]},$$
which shows (\ref{eq:to-prove}) and hence concludes the proof of the theorem. 
\end{pfof}

\vspace{0.25cm}

Roughly speaking, the sets $\Gamma$ constructed above are the sets of leaves of ``growing trees''.


\subsection{Proof of the Main Theorem for minimal actions}\label{s:proof}

In this section, we prove the statement of the Main Theorem for minimal actions. The proof relies on the dichotomy  
given by Theorem~\ref{t:exp}. More precisely, we will show that inequality (\ref{exp-sums}) forces property $(\star)$ to hold.  
Actually, we will show a much stronger fact, namely, property $(\star)$ holds provided 
the sum of derivatives along balls grows faster than linearly. We state this as a 
proposition for further reference.

\begin{proposition}\label{p:linear}
Let $G \subset \Diff^{\omega}_+(\Sc)$ be a finitely-generated free group acting minimally and 
having a non-expandable point $x_0$. Let $S_n (\cdot)$ be the function defined on the circle as 
$$S_n (x) := \sum_{g \in B(n)} g'(x),$$
and denote $S_n := S_n(x_0)$. If
\begin{equation}\label{e:condition}
\frac{S_n}{n} \longrightarrow \infty \quad \mbox{ as } \quad n \to \infty,
\end{equation}
then the stabilizer of $x_0$ is nontrivial.
\end{proposition}



We prove this by contradiction. Assume that the stabilizer of $x_0$ is trivial.
For each $n \geq 1$, let us consider the set $X_n:=\{g(x_0) \mid g\in B(n)\setminus \{\id\} \}$. 
Let $x_n$ be the point of $X_n$ that is closest to $x_0$ on the right. Then $x_n = f_{n}(x_0)$ 
for a unique $f_n \in B(n)$. Denote $I_n := [x_0,x_n)$.
We start with an elementary

\vspace{0.15cm}

\begin{lemma}\label{l:intersection}
The images $g(I_n)$, where $g\in B([n/2])$, are pairwise disjoint.
\end{lemma}

\begin{proof}
If two such intervals $g_1([x_0,x_n))$ and $g_2([x_0,x_n))$ intersect, then the left endpoint of one of them must 
belong to the other one, say $g_2(x_0) \in g_1([x_0,x_n))$. Therefore, $g_1^{-1}\circ g_2 (x_0) \!\in\! [x_0,x_n)$. 
If $g_1 \neq g_2$, then $g_1^{-1}\circ g_2 (x_0) $ cannot be equal to $x_0$, because of the hypothesis on the 
stabilizer of $x_0$. Nevertheless, this is impossible, because  $g_1^{-1}\circ g_2$ belongs to 
$B(n)\setminus \{\id\}$, and we defined $x_n$ to be the image of~$x_0$ under a map 
$g\in  B(n)\setminus \{\id\}$ that is the closest on the right among all such images.
\end{proof}

\begin{corollary}\label{c:control}
The distortion coefficients of all maps $g\in B([n/2])$ are uniformly bounded on $I_n$.
\end{corollary}

\begin{proof} 
Write $g = \gamma_k\cdots \gamma_1$, so that the intermediate images 
$\gamma_j\cdots \gamma_1 (I_n)$, $j \leq k$, are pairwise disjoint. Then 
the sum of their lengths does not exceed 1. A direct application of 
Proposition~\ref{p:sum} thus concludes the proof.
\end{proof}

\begin{lemma}\label{l:conv-int}
The length $|I_n|$ decreases to zero as $n$ tends to infinity. More precisely, 
there exists a positive constant $C_1$ such that for all $n \geq 1$, 
$$\quad |I_n|\le \frac{C_1}{S_{[n/2]}}.$$
\end{lemma}

\begin{proof}
By Corollary~\ref{c:control}, there exists a constant $c_1>0$ such that for all $g\in B([n/2])$, 
we have $|g(I_n)|\ge c_1 g'(x_0) |I_n|$. Hence,
$$
\sum_{g\in B([n/2])} |g(I_n)| \ge c_1 |I_n| \sum_{g\in B([n/2])} g'(x_0) = c_1 |I_n| S_{[n/2]}.
$$ 
On the other hand, since the intervals in the left-side expression are disjoint, 
the sum of their lengths does not exceed~$1$. Thus,
$$
|I_n| \le \frac{1}{c_1 S_{[n/2]}} ,
$$
which concludes the proof.
\end{proof}

Now, let us fix a sequence of positive numbers $r_n$ such that $r_n = o(1/n)$ and $1/S_{[n/2]}~=~o(r_n)$. For 
instance, $r_n := 1 / \sqrt{n S_{[n/2]}}$ will do, though we do not need to deal with any explicit formula. 
The following lemma says that for all sufficiently large~$n$, every map $g\in B(n)$ is ``almost 
affine'' in a complex neighborhood of $x_0$ of radius~$r_n$.

\begin{lemma}\label{l:dc2} There exist $C_2 >0$ and an integer $n_1$ such that for all $n\ge n_1$ 
and all $g\in B(n)$, one has 
$$\varkappa \Big( g, U_{r_n}^{\bbC}(x_0) \Big) \le C_2 n r_n \xrightarrow[n\to\infty]{} 0.$$
\end{lemma}

\begin{proof}
As the point $x_0$ is non-expandable,  
writing each $g \in B(n)$ in its reduced form $g~\!=\!~\gamma_k\cdots \gamma_1$, 
$k\le n$, we obtain the following upper bound for the sum of the intermediate derivatives:
$$
\sum_{j=0}^{k-1} (\gamma_j\cdots \gamma_1)'(x_0) \le \sum_{j=0}^{k-1} 1 = k \le n.
$$
A direct application of Proposition~\ref{bound-complex} then shows the lemma.
\end{proof}

The next two lemmas are devoted to show that the maps $f_n$ are in 
fact close to the identity on the $r_n$-neighborhood of~$x_0$.

\begin{lemma}\label{l:id}
The maps $\widetilde{f}_n(y) := \frac{1}{r_n} \Big( f_n \big( x_0+r_ny \big) - x_0 \Big)$ 
converge to the identity in the~$C^1$ topology on~$U_1^{\bbC}(0)$.
\end{lemma}

\begin{proof}
By Lemma \ref{l:conv-int}, we have that 
$$\widetilde{f}_n(0) = \frac{1}{r_n} (x_n-x_0) = O \Big( \frac{1}{r_n S_{[n/2]}} \Big) 
\xrightarrow[n \to \infty]{} 0.$$ 
Hence, it suffices to check that the derivatives $\widetilde{f}_n'$ tend to~$1$. Moreover, due to the control 
of distortion guaranteed by Lemma~\ref{l:dc2}, it suffices to check such a convergence at a single point.

To do this, let us consider the map $f_n^{-1}\in B(n)$. On the one hand, 
$(f_n^{-1})'(x_0)\le 1$, as $x_0$ is a non-expandable point. On the other hand, 
$$
(f_n^{-1})'(x_n)= (f_n^{-1})'(f_n(x_0)) = \frac{1}{f_n'(x_0)} \ge 1.
$$
Since $|[x_0,x_n]| = O(1 / S_{[n/2]}) = o(r_n)$, for all $n$ sufficiently large, we have 
$x_n\in U_{r_n}(x_0)$. Control of distortion then implies that these 
two derivatives are close to each other:
$$
\left| \log \frac{(f_n^{-1})'(x_n)}{(f_n^{-1})'(x_0)} \right| \le C_2 n r_n.
$$
Thus, both derivatives are $o(1)$-close to~$1$, and hence the same holds for $f_n'(x_0)$. 
\end{proof}

Next, in order to apply Ghys' commutators technique, we need to find \emph{two} elements in $G$ 
that generate a non-solvable group and are close enough to the identity on some interval. 
This is done by the following

\begin{lemma}\label{l:free-gen}
The sequence $f_n$ contains an infinite subsequence $f_{n_i}$ such that for 
each~$i$, the elements $f_{n_i}$ and $f_{n_{i} +1}$ generate a free group.
\end{lemma}

\begin{proof}
It is well known that any two elements of a free group either generate a free subgroup, or belong to the same 
cyclic subgroup. Thus, if the conclusion of this lemma didn't hold, all the maps $f_n$ with sufficiently 
large $n$ would be powers of some fixed $h \in G$, say $f_n = h^{k(n)}$ for all $n \geq m$. 
Since $f_n$ belongs to $B(n)$, we must necessarily have $|k(n)|\le n$.

Now for each $n\ge m$, let us consider the segment of $h$-orbit 
$Y_n:=\{h^j(x_0) \mid j=-n,\dots,n\}$. Let $\widetilde{x}_n$ be the point of $Y_n\setminus \{x_0\}$ 
that is closest to $x_0$ on the right. The interval $\widetilde{I}_n := [x_0,\widetilde{x}_n)$ is contained 
in $I_n$. As a consequence, $x_0$ is (nonperiodic and) recurrent under the action of $h$, hence the 
rotation number $\tau(h)$ must be irrational.

Next, notice that there are arbitrarily large values of $n$ with the following property: the 
intervals $h^j (\widetilde{I}_n)$, $j = 0,1,\ldots,2n$, cover the whole circle with multiplicity 
at most two. Indeed, such a property is invariant under topological 
conjugacy, thus it suffices to check it for the Euclidean rotation of angle 
$\tau(h)$. Now, for this particular case, it can be easily verifed for each integer of the form 
$n=q_{i}-1$, where $\frac{p_i}{q_i}$ is the sequence of good rational approximations of $\tau(h)$.

Finally, control of distorsion shows that for each of the integers $n$ above and $0 \leq j \leq 2n$,
$$|h^j(\widetilde I_n)| \le C (h^j)'(x_0) |\widetilde{I}_n| \le C |\widetilde{I}_n| \le \frac{CC_1}{S_{[n/2]}},$$
where the second inequality comes from that $x_0$ is non-expandable and the last one from Lemma 
\ref{l:conv-int}. As a consequence, the sum of lengths of the intervals $h^j (\widetilde{I}_n)$, 
$j = 0,1,\ldots,2n$, is $O(n/S_{[n/2]})$. It is hence smaller than~$1$ for $n$ 
large enough, which contradicts the fact that these intervals cover the circle. 
\end{proof}

Together with Lemma~\ref{l:id}, the preceding lemma implies that there exists $n$ such that the maps 
$f_n$, $f_{n+1}$ generate a free group and are simultaneously close to the identity 
on $U_{r_n}^{\bbC}(x_0)$. By Proposition~\ref{p:to-id}, the sequence 
of their commutators tends to the identity on~$U_{r_n/2}^{\bbC}(x_0)$. 
Using Proposition~\ref{p:affine}, we conclude that there exist intervals $J\subset I\subset \Sc$ 
and a change of variables  $\varphi:I\to [-1,2]$ after which the $C^1([0,1], [-1,2])$-closure 
of the set of restrictions $\{\varphi \circ g|_I \circ \varphi^{-1} \mid g\in G\}$ contains all 
the translations $T_s:y\mapsto y+s$ with $s\in [-1,1]$. 

We claim that the last property above yields a contradiction. To see 
this, first recall that Sacksteder's theorem is valable in our context: there exists $f \in G$ with a 
hyperbolic repelling fixed point $y_0$ (see \cite{DKN} for a discussion on this). By minimality, 
$y_0$ can be chosen to belong to any small open subinterval (up to changing $f$ by 
an appropriate conjugate). Now, 
the existence of translations in the local $C^1$ closure of the group implies that $x_0$ can 
be mapped arbitrarily close to $y_0$ keeping the derivative bounded away from zero: 
$$
\exists C_3>0, \, h_n \in G : \quad h_n (x_0) \to y_0, \quad  \quad h_n'(x_0) \ge C_3 \quad \forall n. 
$$
Take $k$ such that $f'(y_0)^k > \frac{1}{C_3}$, and consider the element 
$g_n := f^k \circ h_n$. We have 
$$\limsup_{n \to \infty} g_n' (x_0) = 
\limsup_{n\to \infty} (f^k \circ h_n)'(x_0) \ge C_3 \cdot f'(y_0)^k >1,
$$
which contradicts the fact that $x_0$ is non-expandable. 
This closes the proof of Proposition~\ref{p:linear}, hence that of the Main Theorem assuming Theorem~\ref{t:exp}.

\vspace{0.1cm}

\begin{remark}\label{r:NE-C1}
The final arguments of this proof, together with Remark~\ref{r:C1}, imply that 
whenever $\NE \neq \emptyset$, the group $G$ must be $C^1$-locally discrete.  
\end{remark}

\begin{remark} In the same spirit of the preceding remark, there is an alternative end of proof for Lemma 
\ref{l:free-gen}. Namely, after having detected the element $h \in G$ of irrational rotation number, 
we may invoke a theorem of Herman from \cite{Herman}, according to which the sequence of diffeomorphisms 
$h^{q_i}$ converges to the identity in the $C^1$ topology (see \cite{navas-triestino} for a much shorter proof of 
this fact). Now, using this sequence and Remark~\ref{r:C1}, we may conclude as above that there is a sequence 
of elements $g_n \in G$ for which $\lim_{n \to \infty} g_n' (x_0) > 1$, which is a contradiction.
\end{remark}


\subsection{A complementary remark}

\emr{For further use, we next deal with the case of a point $x_0 \! \in \! \NE$ having nontrivial stabilizer. We will show that, in 
this case, the sum of the derivatives at $x_0$ of elements in a ball of radius $n$ cannot grow faster than quadratically on $n$. }

\emr{Formally speaking, we do not need such a consideration for the proof of our Main Theorem: if all the 
points in NE have nontrivial stabilizers, we already have the property ($\star$). However, this consideration 
is interesting by itself. For instance, together with Example~\ref{ex:psl} (this shows that quadratic growth is critical 
for this phenomenon), it answers the question ``how fast such sums may grow?''. Besides, it turns out to be quite 
useful. Namely, as the reader will see in Remark~\ref{r:rare} below, it allows to construct a Markov partition for an 
action that uses only the generators of the group (this argument is one of the key points of~\cite{A-etal}).
}

\begin{proposition}\label{p:quadratic}
Let $G \subset \Diff^{\omega}_+(\Sc)$ be a finitely-generated free group acting minimally and 
having a non-expandable point $x_0$. As before, denote 
$$S_n := S_n (x_0) = \sum_{g \in B(n)} g'(x_0).$$
If $x_0$ has nontrivial stabilizer, then the following growth estimate cannot hold:
\begin{equation}\label{e:condition-q}
\frac{S_n}{n^2} \longrightarrow \infty \quad \mbox{ as } \quad n \to \infty.
\end{equation}
\end{proposition}

For the proof, we first need some information on the stabilizer of $x_0$. 

\vspace{0.1cm}

\begin{lemma} The stabilizer of $x_0$ in $G$ is infinite cyclic.
\end{lemma}

\begin{proof} By a lemma of I.~Nakai (see~\cite[Section~3]{Nakai}), if this stabilizer was not Abelian, then it would 
contain a flow in its closure. However, this is impossible by the very same reasons of those at the 
end of the proof  in the preceding case. Hence, the stabilizer of~$x_0$ is Abelian, and since $G$ is 
free, it must be cyclic.
\end{proof}

Let us denote by $\tilde h$ the generator of the stabilizer of $x_0$ that is topologically 
contracting towards $x_0$ on a right neighborhood of it. Notice that $\tilde h'(x_0) =1$, 
as $x_0$ is non-expandable. Again, for each $n \geq 1$, let us 
consider the set $X_n:=\{g(x_0) \mid g\in B(n) \}$, let us denote by $x_n$ 
the point of $X_n \setminus \{x_0\}$ that is closest to $x_0$ on the right, and let $I_n:=[x_0,x_n)$. 
In what follows, we will assume that $n$ is large enough so that $I_n$ contains no 
fixed point of $\tilde h$ in its interior (this is possible as the orbit of $x_0$ is dense). 
Although the images $g(I_n)$, where $g\in B([n/2])$, are no longer pairwise 
disjoint, we have the next

\vspace{0.1cm}

\begin{lemma}\label{l:cover}
If $g_1,g_2$ in $B([n/2])$ are such that $g_1 (I_n)$ and $g_2 (I_n)$ do intersect, 
then there exists $j$ such that $g_2 = g_1 \tilde h^j$ and $|j| \leq n$.
\end{lemma}

\begin{proof}
The arguments of the beginning of the proof of Lemma \ref{l:intersection} show that $g_2 = g_1 \tilde f$ 
for a certain $\tilde f$ in the stabilizer of $x_0$. Writing $\tilde f$ as $\tilde h^j$ for a certain $j$, we have 
that $\tilde h^j = g_1^{-1} g_2$ belongs to $B(n)$, which forces $|j| \leq n$.
\end{proof}

\vspace{0.1cm}

\begin{lemma}\label{l:decreasing}
The length $|I_n|$ decreases to zero as $n$ tends to infinity. More precisely, 
there exists a constant $C_1'$ such that for all $n \geq 1$, 
$$|I_n| \leq \frac{C_1' n}{S_{[n/2]}}.$$
\end{lemma}

\begin{proof} Let $r_n' > 0$ be such that $r_n' = o(1/n)$ and 
$n/S_{[n/2]}~=~o(r_n')$. (For instance, $r_n' := 1 / \sqrt{S_{[n/2]}}$ will do.)
Let $I_n' := [x_0,x_n')$ be the interval to the right of $x_0$ of length exactly equal to 
$r_n'$. The argument of the proof of Lemma \ref{l:dc2} yields constants $C_2', \bar{C}_2'$ 
such that each $g\in B(n)$ satisfies 
\begin{equation}\label{e:dist}
\varkappa(g, I_n') \le \bar{C}_2' n r_n' \leq C_2'.
\end{equation} 
By the preceding Lemma, the multiplicity of the family of intervals $g(I_n)$, 
with $g$ ranging over $B([n/2])$, is at most $2n+1$. Therefore, 
$$
2n+1 \geq \sum_{g\in B([n/2])} |g(I_n)|.
$$ 
Assume that $I_n$ is not contained in $I_n'$. Then $g(I_n') \subset g(I_n)$ for all $g$, which yields
$$2n+1 \geq \sum_{g\in B([n/2])} |g(I_n')|,$$
Using (\ref{e:dist}) we thus obtain 
$$2n+1 \geq \sum_{g\in B([n/2])} e^{-C_2'} g'(x_0) |I_n'| = e^{-C_2'} |I_n'| S_{[n/2]}.$$
As a consequence, 
$$r_n' = |I_n'| \leq \frac{e^{C_2'}(2n+1)}{S_{[n/2]}},$$
which is impossible for large $n$ due to our assumption (\ref{e:condition}). 
Therefore, $I_n$ is contained in $I_n'$ for large $n$, which easily allows 
to show that its length goes to zero with the claimed speed. 
\end{proof}

\vspace{0.1cm}

Starting from this point, the proof of Proposition \ref{p:quadratic} proceeds 
along the lines of that given for Proposition \ref{p:linear}. 
Namely, assuming that (\ref{e:condition-q}) holds, we now 
work on neighborhoods about $x_0$ of radius $r_n'$ instead of~$r_n$. Moreover,
although the elements $f_n \in B(n)$ sending $x_0$ into $x_n$ are no longer uniquely 
defined, any element in $B(n)$ sending $x_0$ into~$x_n$ will actually work. 
The final outcome will be the existence of $g \in G$ such that $g'(x_0) > 1$, which is 
contrary to our hypothesis.
We leave the details to the reader.

\vspace{0.21cm}

Due to the next example, it seems that non super-quadratic growth is close to be a sharp 
condition for the validity of Proposition \ref{p:quadratic}.

\begin{example} {\em Let us consider the projective minimal 
(real-analytic) action of $\mathrm{PSL}_2 (\mathbb{Z})$ on the circle identified 
with the projective line (see \cite[\S 5.2]{DKN-MMJ} for details). The point $(1:0)$ is 
non-expandable for this action. Moreover, the derivative at this point of an element 
$\left[\left(\begin{smallmatrix} a& b\\ c& d\end{smallmatrix}\right)\right] \in \mathrm{PSL}_2 (\mathbb{Z})$ 
is $1/(a^2 + c^2)$. Now, if such an element belongs to the ball of radius $n$ with respect to 
the system of generators   
$\left[\left(\begin{smallmatrix} 1 & 1\\ 1 & 0\end{smallmatrix}\right)\right], 
\left[\left(\begin{smallmatrix} 1 & 0\\ 1 & 1\end{smallmatrix}\right)\right]$, 
then an elementary argument shows that the absolute values of the entries $a,b,c,d$ 
are smaller than or equal to the $n^{th}$ term $F_n$ of the Fibonacci sequence. Taking 
into account the action of the stabilizer of $(1:0)$, this roughly yields the upper bound 
$$n \sum_{|a| \leq F_n, |c| \leq F_n} \frac{1}{1 + a^2 + c^2} \sim n \log(F_n) = O(n^2)$$
for the sum of derivatives of elements in $B(n)$ at this point. Actually, a finer argument shows that this 
sum is $o(n^2)$: this comes from the fact that the law of elements of continued fractions (Gauss-Kuzmin 
distribution; see, e.g.,~\cite{Ustinov} and references therein) has an infinite expectation. On the other hand, 
it seems quite plaussible that the corresponding averages grow at most logarithmically, and thus that the 
sums $S_n((1:0))$ for $\mathrm{PSL}_2(\mathbb{Z})$ are bounded from below by $n^2/\log^2 n$.

Finally, $\mathrm{PSL}_2(\mathbb{Z})$ contains an index-6 free subgroup (in two generators). 
For this subgroup, the same arguments apply.} \label{ex:psl}
\end{example}

\begin{remark} \label{r:rare}
The more involved argument above for points $x_0 \in \NE$ with nontrivial stabilizers 
allows us to obtain a direct proof for the existence of a Markov partition. Indeed, as we will show in 
the next Section, the failure of~(\ref{exp-sums}) not only leads to property $\smin$ with a 
nonempty set $\NE$, but also to the existence of a Markov partition for the dynamics. Moreover,
this Markov partition only uses the generators and their inverses, and in this sense it is better 
than the one constructed in~\cite{FK1} as a consequence of property~$\smin$ whenever 
$\NE \neq \emptyset$.

\end{remark}


\subsection{Growth of sums of derivatives along geodesics}\label{s:cones}

This section is devoted to the proof of Proposition~\ref{p:geod}. In fact, we will prove a stronger statement. 
\emr{To do this, first recall that
$$
\mC_{\gamma} = \left\{ \gamma_k \cdots \gamma_1 \mid 
\gamma_1 = \gamma \right\}.
$$
Next, for each $\gamma \in \mathcal{G}$, consider the function $S_{\gamma}:\Sc \to [0, +\infty]$ defined as
$$
S_{\gamma}(y) := \sup \left\{ \Sf(g,y) \mid y\in \mC_{\gamma} \right\},
$$
where 
\begin{equation}\label{eq:s-hat}
\Sf(g,y) := \sum_{j =0}^{k-1} (\gamma_j \cdots \gamma_1)'(y), \quad g=\gamma_k\dots\gamma_1,
\end{equation}
is the sum of the intermediate derivatives at $y$ along the geodesic $g$.}

\vspace{0.1cm}

\begin{proposition}\label{p:suma-divergente}
Let $G \subset \Diff^{\omega}_+(\Sc)$ be a finitely-generated free group acting minimally and 
having non-expandable points. If there exist $\gamma \in\mathcal{G}$ and $y\in\Sc$ such that 
$S_{\gamma} (y)<\infty$, then $G$ satisfies property~\smin.
\end{proposition}

\vspace{0.1cm}

For the proof, first notice that control of distortion arguments ({\em c.f.,} Proposition~\ref{bound}) easily yield that the function $S_{\gamma}$ is continuous and the set 
$$M_{\gamma} := \left\{ y \in \mathbb{S}^1 \mid S_{\gamma} (y) < \infty \right\}$$ 
is open, for every 
$\gamma \in \mathcal{G}$. In the sequel, we will actually be mostly concerned with the functions 
\emr{$$
\bS_{\gamma} (y):=\max_{\widetilde{\gamma} \neq \gamma^{-1}} S_{\widetilde{\gamma}}(y)= \sup \left\{ \Sf(g,y) \mid g \in \tC_{\gamma} \right\},
$$ 
where 
$$
\tC_{\gamma} := \left\{ \gamma_k \cdots \gamma_1 \mid 
\gamma_1 \neq \gamma^{-1} \right\} = \bigcup_{\widetilde{\gamma} \neq \gamma^{-1}} \mC_{\widetilde{\gamma}},
$$}as well as with the sets 
$$
\widetilde{M}_{\gamma}:= \left\{ y\in\Sc \mid \bS_{\gamma}(y)<\infty \right\} 
= \bigcap_{\widetilde{\gamma}\neq \gamma^{-1}} M_{\widetilde{\gamma }}.
$$

\begin{example}\label{ex:torus}
{\em The free group in two generators may be seen as the Fuchsian group associated to an hyperbolic punctured torus. 
A fundamental domain of its action on the hyperbolic disc is an absolute quadrangle, and the maps $f$ and $g$ are respectively 
gluing its opposite sides (see Fig.~\ref{f:torus} on the right). The vertices of this quadrangle divide the absolute circle into four arcs.\\ 
\indent The dynamics of the action of $\langle f,g \rangle$ on the absolute circle is Shottky-like: each map in 
$\mathcal{G} := \{ f,g,f^{-1},g^{-1} \}$ sends three of these arcs into one of them and the other one into 
the remaining three. It is thus very natural, and not so difficult to see, that these four arcs turn out to be 
$\bM_f, \bM_g, \bM_{f^{-1}}$ and $\bM_{g^{-1}}$. More precisely, the arc into which $\gamma \in \mathcal{G}$ 
contracts three of them is~$\bM_{\gamma}$. Indeed, if we apply any geodesic that does not start with 
$\gamma^{-1}$, we contract the corresponding arc further and further; moreover, one can check (though this is 
not an immediate computation) that the associated sum of derivatives stay bounded for any point inside the arc.}
\end{example}

\vspace{0.1cm}

For the rest of this section, the goal is to show that the situation in the general case is quite alike the one 
that we have just described in the preceding example. More precisely, the sets $\bM_{\gamma}$ will turn out 
to be unions of finitely many intervals, so that they make a finite partition of the circle, and applying 
$\gamma^{-1}$ on each $\bM_{\gamma}$ will lead to an expansion-like Markovian dynamics. 
In this way, the dynamics of the action of $G$ will turn out to be quite similar to that of the ``degenerated'' 
Schottky group $\mathbb{F}_2 \subset \PSL_2(\bbZ)$ described above.

\begin{figure}[h]
\begin{center}
\includegraphics[scale=0.6]{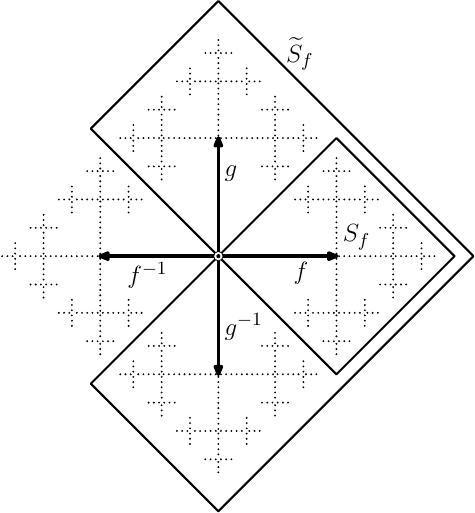}\qquad 
\hspace{0.5cm}
\includegraphics{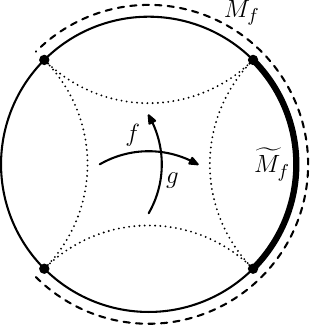}
\end{center}
\caption{On the left: the possible directions of geodesics for $S_f$ and $\bS_f$ in the free group $\langle f,g \rangle$. 
On the right: the sets $M_f$ and $\widetilde{M}_f$ for a Fuchsian group corresponding to a punctured torus.}\label{f:torus}
\end{figure}

\begin{lemma}\label{l:images} The following properties hold:
\begin{enumerate} 
\item For all $\gamma\in\mathcal{G}$, one has $M_{\gamma}\cap \bM_{\gamma} = \bigcap_{\widetilde{\gamma}\in\mathcal{G}} 
M_{\widetilde{\gamma}} = \emptyset$; 
also, for any two different $\gamma_1, \gamma_2$ in $\mathcal{G}$, the sets 
$\bM_{\gamma_1}$ and $\bM_{\gamma_2}$ are disjoint. 
\item For all $\gamma \in\mathcal{G}$, the image $\gamma(M_{\gamma})$ coincides with $\bM_{\gamma}$.
\item\label{i:img} For all $\gamma_1,\gamma_2$ in $\mathcal{G}$ such that $\gamma_1\neq \gamma_2^{-1}$, 
one has $\gamma_1(\bM_{\gamma_2})\subset \bM_{\gamma_1}$.
\item If at least one of the sets $M_{\gamma}$ is nonempty, then all sets 
$\bM_{\gamma}$ are nonempty.
\end{enumerate}
\end{lemma}

\begin{proof}
The first property follows from Proposition~\ref{p:geodesic-free}, 2) and 3) are immediate 
consequences of the definition, and 4) follows as a combination of them.
\end{proof}

\indent Take any connected component $I$ of a set $\bM_{\gamma}$. We will consider the dynamics of 
$I$ under the action of~$G$. We start by studying how the boundedness of $\bS_{\gamma}$ disappears 
at the endpoints of~$I$. To do this, we will consider the images of $I$ until the ``first-return''.

\begin{definition} Let $\gamma \in \mathcal{G}$ and $I$ as above. An element $g \in G$ will be said to be 
{\em $(\gamma,I)$-admissible} if it may be written in the geodesic form $g=\gamma_n\cdots\gamma_1$, 
where $\gamma_1\neq \gamma^{-1}$ and all the intermediate 
images $\gamma_k\cdots\gamma_1(I)$, $k=1,\ldots,n-1$,  are disjoint from~$I$. 
A $(\gamma,I)$-admissible element $g$ is a $(\gamma,I)$-\emph{first-return} if $g(I)$ intersects~$I$.
\end{definition}

In what follows, the generator $\gamma$ and the interval $I$ will be fixed. 
Accordingly, we will just speak about admissible elements and first-return maps. 

\begin{lemma}\label{l:disj}
The images of $I$ under any two different admissible elements are pairwise 
disjoint. Moreover, if $g = \gamma_n \cdots \gamma_1$ is a first-return, 
then $\gamma_n = \gamma$, and $g(I)$ is a subset of $I$.
\end{lemma}

\begin{proof}
Assume that the images of $I$ under two admissible elements 
$g = \gamma_n \cdots \gamma_1$ and $\overline{g} = \overline{\gamma}_{m} \cdots\overline{ \gamma}_1$ 
intersect. By Lemma~\ref{l:images}.\ref{i:img}), we have $g (I)\subset \bM_{\gamma_n}$ and 
$\overline{g}(I) \subset \bM_{\overline{\gamma}_{m}}$. Hence, $\bM_{\gamma_n} \cap 
\bM_{\overline{\gamma}_{m}} \neq \emptyset$, which by Lemma~\ref{l:images}.1) implies that 
$\gamma_n = \overline{\gamma}_{m}$. Therefore, we can remove these two letters, thus obtaining 
shorter admissible elements sending $I$ into non-disjoint intervals. This process stops when one 
of the words becomes empty. If the other word becomes empty simultaneously, then they are 
equal, that is, $g = \overline{g}$. Otherwise, a prefix of one of these words is a first-return, 
and this contradicts the definition of an admissible element.

Now, if $g=\gamma_n\cdots\gamma_1$ is a first-return, then we have 
$\bM_{\gamma_n}\bigcap \bM_{\gamma} \supset g(I)\bigcap I \neq \emptyset$. 
By Lemma \ref{l:images}.3), we must necessarily have $\gamma_n = \gamma$. Finally, 
since the interval $g(I)$ is a subset of $\bM_{\gamma}$ and $I$ is a connected component of 
$\bM_{\gamma}$, we must also have $g(I) \subset I$.
\end{proof}

\vspace{0.1cm}

We next come up with the key step of the proof.

\vspace{0.2cm}

\begin{lemma}\label{l:ends}
There exist first-returns $g_-,g_+$ that fix respectively the left and 
right endpoints $x_{-},x_{+}$ of $I$. 
\end{lemma}

\vspace{-0.432cm}

\emr{
\begin{lemma}\label{l:different}
The elements $g_-$ and $g_+$ above are different.
\end{lemma}
}

\vspace{0.2cm}

Before passing to the proofs, let us notice that this is exactly what happens for the Fuchsian 
group corresponding to a punctured torus:
\begin{example}\label{ex:g-plus}
{\em Let $G = \langle f,g \rangle$ be the group of Example~\ref{ex:torus}. 
Then, the left and right endpoints of $I := \bM_f$ are fixed by $fgf^{-1}g^{-1}$ 
and $fg^{-1}f^{-1}g$, respectively (see Fig.~\ref{f:tree}).}
\end{example}

\begin{figure}[h]
\begin{center}
\includegraphics[scale=0.6]{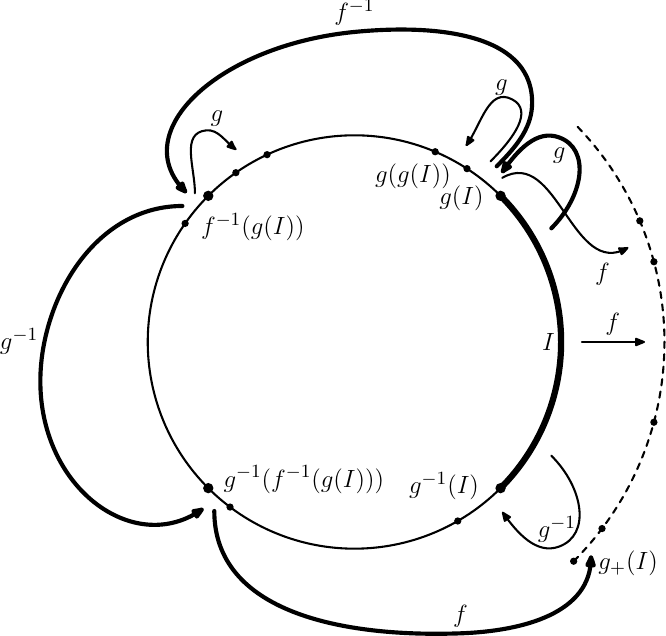} \qquad
\includegraphics[scale=0.6]{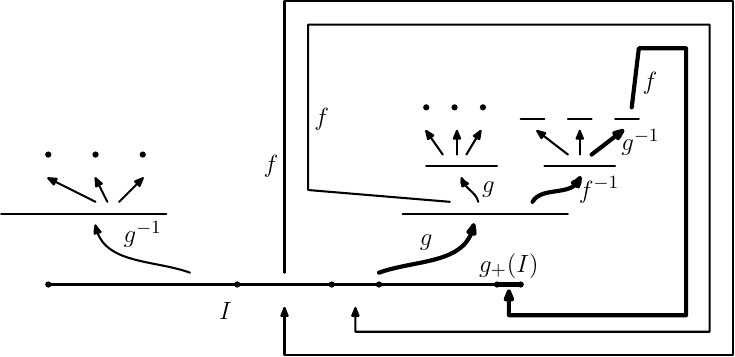}
\end{center}
\caption{On the left: (some) of the admissible images and first returns of the interval $I := \bM_{f}$ for the Fuchsian group 
considered in Examples~\ref{ex:torus} and~\ref{ex:g-plus}. On the right: the corresponding abstract tree of (disjoint!) 
admissible images. In both pictures, the composition $g_+$ is shown by bold arrows.}\label{f:tree}
\end{figure}

\begin{proof}[Proof of Lemma~\ref{l:ends}]
Let us show that there exists a first-return $g_+$ that fixes $x_+$. To do this, first notice that due 
to Proposition \ref{p:sum} and Lemma~\ref{l:disj}, we have a uniform control for the distortion on
$I$ of all admissible words: there exists $C_3>0$ such that for every admissible word $g$, one has 
$\varkappa(g; I)<C_3$. Moreover, the sum of intermediate derivatives corresponding to every 
admissible word at all points of $I$ can be bounded from above by $C_4:=\frac{1}{|I|} e^{C_3}$.

Since the images of $I$ under first-returns are pairwise disjoint, there is only a finite number of 
first-returns $g$ such that $|g(I)|\ge \frac{1}{2}e^{-C_3} |I|$. Let $g_1,\dots,g_m$ be the set 
of all these first-returns. We will show that in fact one of these elements fixes~$x_+$. The proof 
proceeds by contradiction: assuming otherwise,  we will show that $\bS_{\gamma}(x_+)$ is finite.

From now on, assume that none of $g_1,\ldots,g_m$ fixes $x_+$, and consider the sets 
$$
U:= \bigcup_{j=1}^m g_j (\overline{I}), 
\qquad 
\tU:= \bigcup_{j=1}^m g_j (\overline{I}\setminus U)\subset U.
$$
Notice that the set $\tU$ is bounded away from the endpoints of $I$. Indeed, $\tU$ is bounded 
away from the right endpoint $x_+$ by our assumption. Next, either for each $j$ we have 
$g_j (x_-)\neq x_-$, and then the same applies to the left endpoint; or for some~$j$ one 
has $g_j (x_-)=x_-$, and in this case by definition the set $\tU$ cannot intersect 
$g_j^2(\overline{I})~=~[x_-, g_j^2(x_+)]$, hence it is also bounded away from~$x_-$.

Now, since $\tU$ is bounded away from the endpoints of $I$, there 
must exist a finite constant $C_S$ bounding $\bS_{\gamma}$ on $\tU$:
$$
C_S:=\sup\nolimits_{x\in \tU} \bS_{\gamma}(x) < \infty.
$$
We will next prove that for every element $g = \gamma_n \cdots \gamma_1  \, \emr{\in \tC_{\gamma}}$, 
the sum of the intermediate derivatives at each point of $\overline{I}\setminus U$ (including $x_+$) 
does not exceed some constant, thus yielding a contradiction. More precisely, we will show that
\emr{
$$
\Sf(g,y) \le 2\max \big\{ e^{C_3} \cdot C_S,C_4 \big\}, \quad \forall y\in \overline{I}\setminus U, \quad \forall g\in\tC_{\gamma}.
$$
}To do this, we proceed by induction on the length of $g$. The case $n=1$ is evident, as well as 
the case where $g$ is an admissible element. Next, if $g$ is not an admissible element, then it 
contains a prefix $\overline{g} = \gamma_k \cdots\gamma_1$ that is a first-return. Now notice that 
\emr{
\begin{equation}\label{e:decomposition-1}
\Sf(g,y)=\overline{g}'(y) \cdot 
\Sf(\gamma_n\dots\gamma_{k+1}, \overline{g}(y))  + \Sf(\overline{g},y).
\end{equation}
}Since $\overline{g}$ is a first-return, the second \emr{summand} in the right-side expression above 
does not exceed~$C_4$. To estimate the first \emr{one}, we need to consider two possibilities. 

The first possibility is that $\overline{g}$ coincides with one of the $g_j$'s. In this case, we notice that 
$\overline{g}(y)\in\tU$, and hence $\emr{\Sf(\gamma_n\dots\gamma_{k+1}, \overline{g}(y)) \le C_S}$. We claim that the derivative $\overline{g}'(y)$ 
doesn't exceed $e^{C_3}$, and thus the right-side expression of~\eqref{e:decomposition-1} is bounded by 
$e^{C_3} \cdot C_S + C_4$. Indeed, on the one hand, as $\overline{g} (I) \subset I$, there must be a point 
in $I$ at which the derivative is less than or equal to 1. On the other hand, we have 
$\varkappa (\overline g , I) \leq C_3$. The claim easily follows from this.

The second possibility is that $\overline{g}$ does not coincide with any of the $g_j$'s. 
In this case, we have the following uniform upper bound for its derivative:
$$
\overline{g}'(x) \le e^{C_3} \frac{|\overline{g}(I)|}{|I|} \le \frac{1}{2} \quad \forall x\in I.
$$
As a consequence,
\emr{
$$
\Sf(\gamma_n\dots\gamma_1,y) = \overline{g}'(y)\cdot \Sf(\gamma_n\dots\gamma_{k+1},y) + \Sf(\overline{g},y) \le 
\frac{1}{2}\, \Sf(\gamma_n\dots\gamma_{k+1},y) + C_4.
$$
}Since $\overline{g}(y)\in \overline{I}\setminus U$, the induction hypothesis yields that 
\emr{$\Sf(\gamma_n\dots\gamma_{k+1},y)$}
doesn't exceed \hspace{0.04cm} 
$\max \{ e^{C_3} \cdot C_S,C_4 \}$. Thus, we have obtained the desired upper bound: 
$$
\max \{ e^{C_3} \cdot C_S,C_4 \} + C_4 \le 2 \max \{ e^{C_3} \cdot C_S,C_4 \}.
$$

This concludes the proof by contradiction, as we have shown that $\bS_{\gamma}(x_+)$ is finite. This contradiction 
came from the assumption that no $g_j$ fixes $x_+$. Thus, there exists a first-return $g_+$ among 
$g_1, \ldots, g_m$ that fixes $x_+$. Similarly, there must exist a first-return~$g_{-}$ among 
$g_1, \ldots, g_m$ that fixes $x_{-}$. 
\end{proof}

\vspace{0.2cm}

\emr{In what follows (particularly for the proof of Lemma \ref{l:different}), we will need one more (algebraic) tool. Roughly speaking, as $g_+$ is a map that fixes~$x_+$,
in order to study the dynamics of $G$ near $x_+$ it is worth decomposing the action of an element of $G$ by
 ``extracting'' from it the maximal possible power of~$g_+$. This is done by the following}
 
 \vspace{0.01cm}

\emr{ 
\begin{lemma}\label{l:decomposition}
For every $h\in G$ there exist $k \in \bbZ$ and $\bar{h}\in G$ such that $h=\bar{h} g_+^k$, where $\bar{h}\in \tC_{\gamma}$ does 
not have $g_+$ as a prefix. Moreover, if $h\in \tC_{\gamma}$, then  
we have $k\ge 0$, and if $h\in \tC_{\gamma_1^{-1}}$ (where $g_+ = \gamma_n \cdots \gamma_1$), 
then we have $k\le 0$.
\end{lemma}
}

\begin{proof}
\emr{
Consider the product $h_m := h g_+^m$. On the one hand, if $m$ is positive and large enough, then $h_m$ starts with the same letter 
as $g_+$ does, hence it doesn't start with $\gamma^{-1}$. On the other hand, if $m$ is negative and very small, then $h_m$ starts 
with the same letter as $g_+^{-1}$ does, hence with $\gamma^{-1}$. Let $m'$ be the smallest integer such that $h_{m'}$ does not 
start with $\gamma^{-1}$. Since $h_{m'} = h_{m'-1} g_+$ and $h_{m'-1}$ starts with $\gamma^{-1} = \gamma_n^{-1}$, we have that 
$h_{m'}$ does not have $g_+$ as a prefix. Letting $\bar{h} := h_{m'}$ and $k := -m'$, the equality $h = \bar{h} g_+^k$ holds, and 
$\bar{h}$ satisfies the desired properties.}


\emr{If $h$ does not start with $\gamma^{-1}$, then the integer $m'$ above is nonpositive, hence $k \geq 0$. 
If $h$ does not start with $\gamma_1$, then $h$ does not have $g_+$ as a prefix. As a consequence, the integer $m'$ is nonnegative, 
and therefore $k \leq 0$.} 
\end{proof}

\begin{proof}[Proof of Lemma~\ref{l:different}]
\emr{Assume that  $g_+$ and $g_{-}$ coincide, and denote $g$ this element. Since $g_- (x_-)=x_-$ and $g_+ (x_+ )=x_+$, 
we have $g(I)=I$. Since the image of $I$ by the first returns have pairwise disjoint interior, $g$ is the only possible first return.}

\emr{Let $h \!\in\! G$ be arbitrary, and let $h=\bar{h}g_+^k$ be its decomposition provided by Lemma~\ref{l:decomposition}. 
Then $\bar{h}$ does not start with $\gamma^{-1}$ and does not have $g_+$ as a prefix. As $g_+$ is the only 
possible first return, this implies that $\bar{h}$ is admissible.}

\emr{Finally, let $J \subset I$ be a fundamental domain for the action of $g_+$ on~$I$. We claim 
that $J$ is \emph{wandering} for the action of~$G$, that is, for any $h\in G\setminus \{\id\}$ we have 
$h(J)\cap J=\emptyset$. Indeed, by looking at the representation $h=\bar{h}g_+^k$ above, we see that 
$h(J) = \bar{h}g_+^k (J)\subset \bar{h}(I)$, and hence $h(J)\cap I=\emptyset$ unless $\bar{h}=\id$. 
Now, if $\bar{h}=\id$, then $h=g_+^k$, and due to the choice of~$J$ as a fundamental 
domain for the action of~$g_+$ on $I$, we have $h(J)\cap J=\emptyset$ unless $k=0$.}
\end{proof}

\emr{As the point $x \!\in\! I$ approaches $x_+$, we have $\bS_{\gamma}(x)\to+\infty$. We next show that, roughly 
speaking, the element $g_+$ is ``the only'' reason for such a divergence. To do this, consider the function 
$$
S_{g_+}(x):=\sup\left\{ \Sf(g,x) \mid g\in \tC_{\gamma} \text{ and $g_+$ is not a prefix of $g$} \right\}.
$$}


\begin{lemma}\label{l:sum-finite}
\emr{The value of $S_{g_+}(x_+)$ is finite. Moreover, $S_{g_+}$ is  a bounded continuous function on some neighborhood of~$x_+$.}
\end{lemma}

\begin{proof}[Proof of Lemma~\ref{l:sum-finite}]
\emr{Any $h = \bg_m\dots\bg_1$ with $\bg_1\neq\gamma^{-1}$ 
that does not have $g_+$ as a prefix either is an admissible word  
or can be decomposed as $h = \widetilde{g} g$, where $g$ is a 
first return different from $g_+$ and $\widetilde{g}$ does not start with $\gamma^{-1}$. 
In the former case, we have seen at the beginning of the proof of Lemma  \ref{l:ends} that 
the sum of derivatives along the geodesic everywhere on~$I$ is bounded from above 
by~$C_4 = e^{C_3} / |I|$. 
In the latter case, the point $y := g(x_+)\in I$ must belong to the interval $[g_-(x_+),g_+(x_-)]$, due to the fact that 
$g \neq g_+$ and because the images of $I$ by first returns are disjoint. 
The sum of derivatives along $\widetilde{g}$ at $y$ is then bounded from above by 
$\max_{[g_-(x_+),g_+(x_-)]} \bS_{\gamma}<+\infty$. By control of distortion, we have 
$g'(x_+) \le e^{C_3}$, which finally yields 
$$
S_{g_+}(x_+) \le C_4 + e^{C_3} \max_{[g_-(x_+),g_+(x_-)]} \bS_{\gamma}<+\infty.$$
The claim concerning boundedness and continuity of $S_{g_+}$ on a neighborhood of $x_+$ 
follows from control of distortion arguments in the same way as we already noticed for the functions~$S_{\gamma}$.}
\end{proof}

\vspace{0.1cm}

Recall that the general idea is to show that the sets $\bM_{\gamma}$ decompose the circle into finitely 
many intervals that form a Markov partition for the dynamics. Thus, at the same point where a connected 
component of one of these set ends, a connected component of another (perhaps different) set should start. 
We have already seen such a behavior in Example~\ref{ex:torus}. The following lemma shows that this is always the case.

\vspace{0.1cm}

\begin{lemma}\label{l:lr-r}
Let $I, x_+, g_+$ be as above, and write again $g_+ = \gamma_n \cdots \gamma_1$. Then $x_+$ is the left endpoint of a 
connected component of some $\widetilde{M}_{\widetilde{\gamma}}$, where $\widetilde{\gamma} = \gamma$ if $g_+$ 
is topologically contracting on a right neighborhood of $x_+$, and $\widetilde{\gamma} = \gamma_1^{-1}$ otherwise. 
\end{lemma}

\begin{proof}

\emr{Lemma~\ref{l:sum-finite} implies that there exists a neighborhood $U$ of the point $x_+$ for which 
\begin{equation}\label{eq:C-U}
C_U:=\sup_{y\in U} S_{g_+}(y)<+\infty.
\end{equation}
Let $U_+\subset U$ be a right neighborhood of $x_+$ on which $g_+$ has no fixed points. We next show 
that $U_+\subset \widetilde{M}_{\gamma}$ if $g_+$ is topologically contracting towards $x_+$ in $U_+$, 
and that $U_+\subset \widetilde{M}_{\gamma_1^{-1}}$ otherwise. 
}

\emr{Let us consider the first case. Given a point $x\in U_+$, let $J$ be a fundamental domain for the 
action of~$g_+$ on~$U_+$. For an arbitrary $h\in G$ that does not start with $\gamma^{-1}$, let 
$h = \bar{h}g_+^k$ be its decomposition provided by Lemma~\ref{l:decomposition}. Since $k \geq 0$ 
and $g_+$ is topologically contracting towards $x_+$ on $U_+$, the point $y := g_+^k (x)$ belongs 
to $U_+$, hence the sum of the intermediate derivatives associated to $\bar{h}$ at $y$ is  
bounded from above by $C_U$. Moreover, as $J$ is a fundamental domain, all the images $g_+^i(J)$, 
$i=0,1,\dots,k$, are pairwise disjoint, hence the sum of intermediate derivatives of $g_+^k$ at $x$ 
is bounded from above by some constant~$C_J$. Therefore, the sum of intermediate derivatives associated 
to $h$ at $x$ does not exceed $C_J(1+C_U)$. This yields $\bS_{\gamma}(x)\le C_J(1+C_U)<+\infty$, thus 
$x\in \widetilde{M}_{\gamma}$. Finally, as $x\in U_+$ was arbitrary, we have $U_+\subset \widetilde{M}_{\gamma}$.
}

\emr{In the second case where $g_+$ is topologically repelling on a right neighborhood of~$x_+$, a similar argument 
applies. Indeed, let again $J \subset U_+$ be a fundamental domain for the action of $g_+$ that contains a given point  
$x\in U_+$. For $h\in G$ that does not start with $\gamma_1^{-1}$, let $h=\bar{h} g_+^k$ be its decomposition provided by 
Lemma \ref{l:decomposition}. In this case, we have $k \le 0$. As $g_+$ is topologically repelling on $U_+$ and $J$ is a 
fundamental domain, the intervals $J, g_+^{-1}(J), g_+^{-2}(J),\dots,g_+^k(J)$ are pairwise disjoint. Hence, the sum of 
the intermediate derivatives at $x$ of the iterations of $g_+^{-1}$ is smaller than or equal to the same constant 
$C_J$ above. As a consequence, we have
$$
\bS_{\gamma_1^{-1}}(x)\le C_J (1+C_U)<+\infty,
$$
and since $x\in U_+$ was arbitrary, this implies that $U_+\subset \widetilde{M}_{\gamma_1^{-1}}$.
}
%
\end{proof}

\vspace{0.1cm}

\begin{lemma}\label{l:der=1} 
Neither $g_{-}$ nor $g_+$ have fixed points in the interior of $I$. Moreover, $g_{-}'(x_{-})  = g_{+}'(x_{+}) = 1$.
\end{lemma}

\begin{proof} 
\emr{By Lemma~\ref{l:decomposition}, each $h\in \tC_{\gamma}$ 
can be represented as $\bar{h} g_+^k$, where $k \ge 0$ and $\bar{h}\in\tC_{\gamma}$ does not have $g_+$ 
as a prefix. Hence, the sum of intermediate derivatives along $h$ at $x_+$ equals 
\begin{equation}\label{eq:long-chain}
\Sf(h,x_+)= g_+'(x_+)^k \cdot \Sf(\bar{h},x_+)  + \sum_{i=0}^{k-1} g_+'(x_+)^i \cdot  \Sf(g_+,x_+).
\end{equation}
}


Assuming that $\lambda := g_+'(x_+) < 1$, we will prove that $x_+$ belongs to $\widetilde{M}_{\gamma}$, 
thus yielding a contradiction. To do this, notice that the second summand in the right hand side expression 
of~\eqref{eq:long-chain} does not exceed~$\frac{\Sf(g_+,x_+)}{1-\lambda}$. Moreover, the first summand does 
not exceed $S_{g_+}(x_+)$, which is finite due to Lemma~\ref{l:sum-finite}. We thus get the uniform estimate
$$
\Sf(h,x_+)\le S_{g_+}(x_+)+\frac{\Sf(g_+,x_+)}{1-\lambda}.
$$
As $h\in \tC_{\gamma}$ was arbitrary, this implies that $x_+\in\widetilde{M}_{\gamma}$, 
which is the desired contradiction.

We have hence established that that $g_+' (x_+)  \geq 1$.
Now, if we prove that $g_+$ has no fixed point inside $I$, then the fact that $g_+(I)$ is strictly contained in~$I$ 
will imply that $g_{+}'(x_{+})$ cannot be greater than~$1$, and hence equals~$1$. 

Assume for a contradiction that $g_+(y) = y \in I$. \emr{Take an arbitrary point $x \in [y,x_+]$ that is not fixed by $g_+$. 
We next prove that the sum of derivatives $\Sf(h,x)$ is uniformly bounded on $h\in G$, which is in contradiction with 
Proposition~\ref{p:geodesic-free}.} To do this, notice that since 
\emr{$x\in I$, we have $\bS_{\gamma}(x)<\infty$, hence we only need to consider the case of $h$ starting with 
$\gamma^{-1}$. For such an $h$, Lemma~\ref{l:decomposition} allows us to write it in the form $h=\bar{h}g_+^{-k}$, 
with $k\ge 0$ and $\bar{h}\in \tC_{\gamma}$ not having $g_+$ as a prefix. 
As before, we have
\begin{equation}\label{eq:repelling}
\Sf(h,x)=(g_+^{-k})'(x) \cdot \Sf(\bar{h},g_+^{-k}(x))+ \sum_{i=0}^{k-1} (g_+^{-1})'(x)^i \cdot \Sf(g_+^{-1},g_+^{-i}(x)).
\end{equation}
}

\emr{Let $J$ be a fundamental domain for the action of $g_+$ that contains $x$. The images of~$J$ under 
the iterates of $g_+^{-1}$ are pairwise disjoint, hence we have a uniform control of distortion for them. 
Therefore, the second summand in the right side expression of~\eqref{eq:repelling} does not exceed 
some constant independent of~$k$, and the same holds for the factor~$(g_+^{-k})'(x)$ in the first summand. 
Finally, the image point $g_+^{-k}(x)$ belongs to $[y,x_+]$, which implies that the value of $\Sf(\bar{h},g_+^{-k}(x))$ 
is uniformly bounded from above. Indeed, away from $x_+$, boundedness of  
$\Sf(\bar{h}, \cdot)$ follows from the finiteness of the $\bS_{\gamma}$,  
and in a neighborhood of~$x_+$, it follows from Lemma~\ref{l:sum-finite}. As a consequence, we obtain a 
uniform upper bound for the right hand side expression of~\eqref{eq:repelling}, which is the desired contradiction.
}


Obviously, similar arguments apply to conclude that $g_{-}$ has no fixed point in $I$, 
and~$g_{-} '(x_{-}) = 1$.
\end{proof}

\vspace{0.1cm}

\begin{lemma}\label{l:finite}
There is only a finite number of connected components of sets $\bM_{\gamma}$.
\end{lemma}

\begin{proof}
For each connected component $I$ of some $\bM_{\gamma}$, 
write $g_+ = \gamma_n \cdots \gamma_1$, and 
let $\JR (I)$ be the connected component of $\bM_{\gamma_1}$ 
that contains $\gamma_1 (I)$. Note that $\JR(I)$ and $\gamma_1(I)$ 
have the same right endpoint $\gamma_1(x_+)$. Indeed, since 
$\gamma_n\dots \gamma_2(\JR(I))\subset \bM_{\gamma}$, in the 
case where $\JR(I)$ contained $\gamma_1(x_+)$, we would have that 
$g_+(x_+)=x_+$ belongs to $\bM_{\gamma}$, which is absurd. 

We next show that the $(\gamma_1,\Phi_R (I))$-first-return $\widetilde{g}_+$ that fixes 
the right endpoint of $\JR(I)$ is exactly the conjugate of $g_+$ by $\gamma_1$, that is, 
$$
\widetilde{g}_+=\gamma_1 \gamma_n\dots\gamma_2.
$$
Indeed, this conjugate doesn't start with $\gamma_1^{-1}$ and fixes the right endpoint of $\JR(I)$. 
Hence, it is either $\widetilde{g}_+$, or a higher power of it. But the latter case would imply that 
while applying $\gamma_1 \gamma_n\dots\gamma_2$ to $\gamma_1(x_+)$, we pass through 
$\gamma_1 (x_+)$ at some intermediate step. Removing the last applied letter (which should 
be $\gamma_1$), this would imply that $g_+$ wasn't admissible, which is a contradiction.

By the previous discussion, the map $I \mapsto \JR (I)$ yields a dynamics for which 
every $I$ is contained in a cycle. Moreover, the composition of the maps associated 
to the cycle gives the corresponding element $g_+$. Clearly, any two such cycles 
either coincide or are disjoint.  

Similarly, writing $g_{-} = \overline{\gamma}_{m} \cdots \overline{\gamma}_1$, we may consider the 
connected component $\JL(I)$ of $\bM_{\overline{\gamma}_1}$ that contains $\overline{\gamma}_1 (I)$. This yields a 
new dynamics for which every $I$ belongs to a cycle, and any two such cycles either coincide or are disjoint.  

Now, for every connected component $I$ as above, the intervals $g_+(I)$ and $g_-(I)$ are disjoint and contained in $I$, 
hence the length of at least one of them doesn't exceed $|I|/2$. Since by the preceding lemma the derivative at the 
corresponding endpoint equals~1, the distortion of the corresponding first-return is bounded from below by~$\log 2$. This implies that the sum of the lengths of the image intervals along the composition 
is bounded from below by $\frac{\log 2}{C_{\mathcal{G}}}$. 

All the right cycles are disjoint, and so are all the left cycles. Thus, there is but a finite number of possible right cycles 
with such a sum of lengths, and there is but a finite number of possible left cycles. Therefore, there is but a finite 
number of connected components of the sets $\bM_{\gamma}$.
\end{proof}

\emr{We are now ready to conclude the first part of the proof of 
Proposition~\ref{p:suma-divergente} via the construction of a Markov partition.
Namely, }Lemmas~\ref{l:lr-r} and~\ref{l:finite} imply that 
the set of endpoints 
$N=\bigcup_{\gamma\in\mathcal{G}} \partial \bM_{\gamma}$ is finite, 
and its complement consists of the disjoint union of 
the connected components of the $\bM_{\gamma}$'s. Let 
$\mJ=\{I_1,\dots,I_m\}$ be the family of these connected components. This will be 
a (preliminary) partition that we will use.

Next, to this partition we would like to associate a ``topologically expanding'' map 
that sends each interval of the partition to a finite union of intervals of the partition. 
To do this, we define the map $R:\Sc\setminus N \to \Sc$ by
$$
R|_{\bM_{\gamma}} = \gamma^{-1}|_{\bM_{\gamma}}, \quad \forall \gamma \in\mathcal{G}.
$$
\emr{This choice is quite natural: at each point in the interior of each $\bM_{\gamma}$, geodesics 
starting with~$\gamma^{-1}$ are the only ones along which, eventually, one  may get an expansion.}

Note that for all $i$, the image $R(I_i)$ is  a union of some of the $I_j$'s plus the boundary points 
of adjacent intervals. Indeed, Lemma~\ref{l:images} implies that $R(I_i)$ is a connected 
component of the corresponding set~$M_{\gamma}$ (in particular, its endpoints belong to~$N$).

\emr{To conclude the proof of Proposition~\ref{p:suma-divergente}, we need to study the 
partition $\mathcal{J}$ and the map $R$ in more details. We start by defining}~$N_j$ as the 
set of indeterminacy points for the map $R^j$. \emr{More precisely, we} let $N_0:=N$ and 
$N_{j+1}:=N_j \bigcup R^{-j}(N)$. 


\emr{
\begin{lemma}\label{l:NE-G} 
The following holds:
$NE(G)\subset \bigcup_j N_j \subset G(N).$
\end{lemma}}
\emr{Before going into the proof of this lemma, 
we notice that it allows to conclude the proof of Proposition~\ref{p:suma-divergente} 
(and hence those of Proposition~\ref{p:geod} and Theorem~\ref{t:exp}). Indeed, by Lemma~\ref{l:ends}, every point 
of~$N$ is fixed by the corresponding element $g_+\neq\id$. Hence, every point of~$G(N)$ also has a nontrivial 
stabilizer. By Lemma~\ref{l:NE-G}, this yields the desired property~($\star$) for the group~$G$.}

\emr{In order to prove Lemma~\ref{l:NE-G}, we would like to use the topological expansion of the map~$R$.
Notice that}, due to the Markov property of the partition 
$\mJ$, the map $R^{j}$ sends each connected component of $\Sc\setminus N_j$ onto 
one of the intervals~$I_i$. \emr{Hence, to obtain a lower bound for the derivative of $R^j$ 
in the complement of~$N_j$, it would suffice to obtain an upper bound for its distortion and 
show that the size of the connected components of this complementary set tends to zero. 
The former issue is treated (with a slightly modified statement) in Lemma~\ref{l:part+dist} 
below, and the latter is ruled out by Lemma~\ref{l:diameters}; both are then put together 
in Corollary~\ref{c:needed}, concluding the proof of Lemma~\ref{l:NE-G}.}

\emr{However, the distortion of the iterations of $R^j$ on connected components of the 
complement of~$N_j$ cannot be controlled due to the presence of parabolic points. Indeed, the iterations of~$g_+$ 
have derivative identically equal to~$1$ at the corresponding endpoint $x_+\in N$. We are hence forced 
to modify the above procedure to reobtain control of distortion. The method, allowing to do so}  
(which is similar to the one used in~\cite{DKN-MMJ}) \emr{is to vary the number of iterations by  
considering the first iteration of $R$ under which the point leaves the neighborhoods of the points of~$N$.}

Namely, for each interval $I=(\emr{x_-,x_+}) \in \mJ$, consider the corresponding 
first-returns~$g_+$ and $g_-$ fixing the right and left endpoints of $I$, respectively. Let 
\emr{
$$
J_{(I)}:= (g_-^2(x_+),g_+^2(x_-)), \quad \tN_{(I)}:= 
\left( \bigcup_{i=2}^{\infty} \{g_-^i(x_+),g_+^i(x_-)\} \right) \bigcup \{x_-,x_+\},
$$
and denote by $\tR_{(I)}: I \setminus \tN_{(I)}\to J_{(I)}$  the ``$R$-first-exit to $J_{(I)}$'' (see Figure~\ref{f:tR}):
$$
\tR_{(I)}(y) = R^{\min\{m \mid R^m (y) \in J_{(I)}\}}(y), \quad \forall y \in I \setminus \tN_{(I)}.
$$}

\begin{figure}[!h]
\begin{center}
\includegraphics{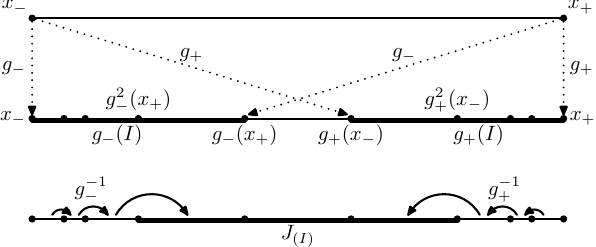}
\end{center}
\caption{The definition of the map $\tR_{(I)}(y)$.}\label{f:tR}
\end{figure}

\emr{
Notice that }
$$
\emr{\tR_{(I)}(y)=\left\{\begin{array}{lll}
g_-^{-i}(y), &\quad & y\in (g_-^{i+2}(x_+),g_-^{i+1}(x_+)),\\
g_+^{-i}(y), &\quad & y\in (g_+^{i+1}(x_{-}),g_+^{i+2}(x_{-})),\\
y, &\quad & y\in J_{(I)},
\end{array}\right.}
$$
\emr{and that $\tR_{(I)}$ sends each connected component of its domain of definition to one of the tree intervals of the set
$$
\mQ_{(I)}:=\{(g_-^2(x_+), g_-(x_+)), (g_-(x_+), g_+(x_-)), (g_+(x_-), g_+^2(x_-)) \}.
$$
}

A crucial fact is that $\widetilde{R}_{(I)}$ locally behaves as an iterate of $R$, with a 
number of iterations that is constant on each interval of the partition $\mJ_j$. 

Next, let us apply the first-exit map after $j$ iterations of $R$. More precisely, 
let $\tR_j:\Sc\setminus \widetilde{N}_j \to \Sc$ be defined as 
$$\tR_j |_{J} = \tR_{(R^{j}(J))} \circ R^j, 
\quad \forall J \in \mJ_j,
$$
where 
$$
\widetilde{N}_j :=  N_j \cup \bigcup_{I \in \mJ} R^{-j}(\tN_{(I)}). 
$$
We let $\widetilde{\mJ}_j$ be the family of connected components of the 
complement of $\tilde{N}_j$.

\begin{lemma}\label{l:part+dist}
For each $j$ and each $J \in \widetilde{\mJ}_j$, the image 
$\tR_j(J)$ is one of the finitely many intervals in the family  
\emr{$$
\mQ:=\bigcup_{I\in\mJ} \mQ_{(I)}.
$$ 
}Moreover, the distortion coefficients are uniformly bounded: 
there exists a constant $C_5$ independent of $j$ such that  
$$\varkappa(\tR_j, J) \le C_5.
$$
\end{lemma}

\begin{proof}
The first claim of the lemma is a direct consequence of the Markov property of~$R$ and the definition 
of $\tR_{j}$. For the second, first write the restriction of $\widetilde{R}_j$ to $J$ in the form $R^k$, 
where $k = k_J \geq j$. Next, notice that the sums of the derivatives along the path $R^{-k}$ starting 
at points in the image interval $\tR_j(J)$ are uniformly bounded. Indeed, this path corresponds to a 
geodesic in $\mC_{\gamma}$, where the interval $I \in \mJ$ containing $\tR_j(J)$ is a connected 
component of $\widetilde{M}_{\gamma}$. Moreover, the image $\widetilde{R}_j (J)$ lies inside an interval of the 
form \emr{$(g_-^2 (x_+), g_+^2(x_-))$}; thus, it is bounded away from the endpoints of $I$, which yields the desired 
uniform bound for the sum of derivatives. Finally, knowing that this sum is uniformly bounded, Proposition 
\ref{p:sum} together with property (\ref{eq:inverses}) guarantee the desired control of distortion.  
\end{proof}

\vspace{0.1cm}

\begin{lemma}\label{l:diameters}
The diameter of the partition $\widetilde{\mJ}_j$ tends to zero as $j \to \infty$.
\end{lemma}

\begin{proof} 
Very close to the right ({\em resp.,} left) endpoint of an interval $I$ in $\mJ$, iterations of~$R^{-1}$ 
correspond to following the path given by the corresponding first-return $g_+$ ({\em resp.,} $g_{-}$). Since 
the maps $g_{-},g_{+}$ are different and the images of $I$ under them are disjoint, this implies that each 
interval in $\mJ$ is eventually divided into at least two subintervals. Actually, according to the construction 
({\em e.g.}, the proof of Lemma \ref{l:ends}),  
each of these intervals retain a proportion of the length of $I$ uniformly bounded from below (at least 
equal to $1/2e^{C_3}$).  Together with the uniform control of distortion provided by the preceding lemma, 
this obviously implies the desired convergence.
\end{proof}

\vspace{0.1cm}

By combining these two lemmas, we finally obtain our desired 

\vspace{0.1cm}

\begin{corollary}\label{c:needed}
There exists $j$ such that $(\tR_j) ' (y)>1$ for all $y \in \Sc\setminus \tN_j$.
\end{corollary}

\begin{proof}
It suffices to take $j$ such that the diameter of the partition $\widetilde{\mJ}_j$ is less than 
$$
\varepsilon_0:=\frac{1}{2}\min_{I \in \mQ} |I| \cdot e^{-C_5}.
$$
Indeed, as $\tR_j(I)\in \mQ$, for all $y\in J \in \widetilde{\mJ}_j$ we have 
$$
\tR_j'(y) \ge \frac{|\tR_j(J)|}{|J|} \cdot 
e^{-\varkappa(\tR_j,J)} \ge \frac{\min_{U \in \mQ} |U|}{\varepsilon_0}  \cdot e^{-C_5} \ge 2,
$$
which shows the lemma. \end{proof}

We can finally conclude the proof of Proposition \ref{p:suma-divergente}, 
hence that of Proposition \ref{p:geod}. Indeed, we have ensured that our procedure of 
``topological expansion'' by iterations of $R$ indeed generates derivatives greater than $1$ everywhere, 
except on the finite set $N$ and its preimages under this procedure. Thus, $\NE(G)\subset G(N)$. Now, 
every point of $N$ is a fixed point of the corresponding $g_+$, due to Lemma~\ref{l:ends}. Hence, every 
point of the orbit $G(N)$ is also a fixed point of the corresponding conjugate  of~$g_+$. 
We have thus established property~\smin{} for the group $G$.


\section{The case of an exceptional minimal set}
\label{s:modifications}

Most of the previous arguments generalize almost immediately to the case of an action with an 
exceptional minimal set, provided we mostly work with points therein instead of the whole circle.  
Below we sketch the involved steps just stressing the points where some significant difference appears. 

Assume first that $G$ is a (finitely-generated, higher-rank) free subgroup of 
$\mathrm{Diff}_+^{\omega} (\mathbb{S}^1)$ admitting an exceptional minimal set $\Lambda$ (the case of 
a non-necessarily free group will be considered later). Again, the main issue consists in proving an analogous 
to Theorem~\ref{t:exp}: either $G$ satisfies property \sLmin, or there exist positive constants $c,\lambda$ 
such that for all $x \in \Lambda$ and all $n \geq 1$,
\begin{equation}\label{eq:sum-exp}
\sum_{g \in B(n)} g'(x) \geq c e^{\lambda n}.
\end{equation}
This will be shown later. For the moment, let us see why the second possibility leads to a contradiction 
whenever the set $\NE \cap \Lambda$ is nonempty. Actually, as in the minimal case, such a contradiction 
will be established under the assumption of super-quadratic growth of sum of derivatives. 

Fix $x_0 \in \mathrm{NE} \cap \Lambda$. As in \S \ref{s:proof}, define 
$x_n = f_n (x_0) \neq x_0$, $f_n \in B(n)$, as being the point 
in~$X_n := \{ g(x_0) \mid g \in B(n) \}$ that is closest to 
$x_0$ on the right if $x_0$ is non-isolated in $\Lambda$ from the right; otherwise, 
consider the point $x_n \neq x_0$ of $X_n$ that is closest on the left to $x_0$. 
Denote by $I_n$ the interval of endpoints 
$x_0,x_n$. Notice that the length of $|I_n|$ converges to zero, though we would like to show that 
this convergence holds at a rate of order $r_n' := n / S_{[n/2]}$, where
$$S_n := \sum_{g \in B(n)} g'(x_0).$$ 
To prove this, we notice again that the 
intervals $g(I_n)$, $g \in B([n/2])$, are ``almost'' disjoint; more precisely, the multiplicity growths linearly 
with $n$. (Compare Lemma \ref{l:cover}.) 
Since $r_n' = o(1/n)$, the distortion coefficient  of such a $g$ on $I_n$ is well behaved, which 
allows establishing the desired rate of convergence. (Compare Lemma \ref{l:decreasing}.)

Complex control of distortion ({\em c.f.,} Proposition \ref{bound-complex}) then shows that for every $g \in B(n)$, 
$$\varkappa \Big( g, U_{r_n'}^{\bbC}(x_0) \Big) \le C_2 n r_n'.$$ 
Therefore, the maps $\widetilde{f}_n(y)  := \frac{1}{r_n'} \big( f_n(x_0 + r_n' y)-x_0 \big)$ 
converge to the identity in the~$C^1$ topology on $U_1^{\bbC}(0)$. (Compare Lemma~\ref{l:id}.)
As in the minimal case ({\em c.f.,} Lemma \ref{l:free-gen}), 
passing to a subsequence, we have that~$f_{n_i}$ and $f_{n_{i+1}}$ 
generate a free group for each~$i$. This allows performing the commutators procedure, thus finding a local flow 
in the closure of the group. Nevertheless, the orbit of $x_0$ under such a flow is an interval, which is absurd 
since it must be contained in $\Lambda$.

\vspace{0.1cm}

We have hence proved that (\ref{eq:sum-exp}) cannot hold for all $x \!\in\! \Lambda$ and 
all $n \geq 1$. Let us next prove that the failure of (\ref{eq:sum-exp}) leads to property \sLmin. 
As we will see, the proof is not a direct translation of the arguments given for the minimal case. 
Several modifications are needed, mainly because we must concentrate on points of $\Lambda$, 
and not of arbitrary points in the circle. We proceed in several steps, invoking the analogue 
statements for the minimal case at each step. 

First of all, the growing trees argument 
of \S \ref{s:growth} shows that the failure of (\ref{eq:sum-exp}) implies that there must exist $x \in \Lambda$ and 
$\gamma \in \mathcal{G}$ such that for every $n \geq 1$ and every geodesic $\gamma_n \cdots \gamma_1$ 
starting with $\gamma_1 = \gamma$,
$$\sum_{i=1}^n (\gamma_i \cdots \gamma_1)'(x) \leq 2.$$
We may hence consider the (continuous) functions 
$S_{\gamma}$ and $\widetilde{S}_{\gamma}$ of \S \ref{s:cones}, as well as the (open) sets 
$M_{\gamma}$ and $\widetilde{M}_{\gamma}$. However, it will useful to consider the intersections 
$$M_{\gamma}^{\Lambda} := M_{\gamma} \cap \Lambda, \qquad 
\widetilde{M}_{\gamma}^{\Lambda} = \widetilde{M}_{\gamma} \cap \Lambda.$$
These sets satisfy analogous relations to those provided by Lemma \ref{l:images}; in particular, they are 
all nonempty. The fact that the intersection of all the $\widetilde{M}_{\gamma}^{\Lambda}$ is empty 
(hence the disjointness of all $\widetilde{M}_{\gamma}^{\Lambda}$) follows from the next analog 
of Proposition \ref{p:geodesic-free}.

\vspace{0.1cm}

\begin{lemma} \label{l:inf-sum}
For each $x \in \Lambda$ there exist geodesics $g = \gamma_n \cdots \gamma_1$ with arbitrarily 
large sum of intermediate derivatives.
\end{lemma}

\begin{proof} Otherwise, as in the proof of Proposition \ref{p:geodesic-free}, the distortion of every 
group element on a small neighborhood $U$ of $x$ would be uniformly bounded. Taking such an 
element $f$ having an hyperbolically repelling fixed point inside (this is guaranteed by the classical 
Sacksteder's theorem), we would thus conclude that the length of the image $f^j (U)$ is greater 
than 1 for $j$ large enough, which is absurd. 
\end{proof}

The reason for dealing only with points in $\Lambda$ is given by the next 

\begin{lemma} \label{l:exterior-components}
Let $J$ be a connected component of $\mathbb{S}^1 \setminus \Lambda$. 
For all but finitely many points $x \in J$, one has
$$\sum_{g \in G} g'(x) < \infty.$$ 
Moreover, this sum uniformly converges on compacts subsets of the complement in $J$ of the finite subset above. 
\end{lemma}

\begin{proof} According to a result of Hector~\cite{hector} (a proof of which may be found in \cite{Navas-cociclo}), 
the stabilizer of $J$ in $G$ is nontrivial and cyclic, 
say generated by $h \in G$. Let $x_1,\ldots,x_k$ be the fixed points of $h$ inside $I$. We claim that the sum 
above is finite for $x \in J$ different from the $x_i$'s. Indeed,  each such point belongs to a wandering 
interval for the action of $G$, namely a fundamental domain for the action of $h$. The lemma then 
follows from a control of distortion argument.
\end{proof}


While applying the arguments of the previous section in the non-minimal case, one should have in 
mind that the sets~$\bM_{\gamma}$ are no longer pairwise disjoint, yet their intersections with~$\Lambda$ 
remain disjoint. At the same time, the endpoints of the sets $\widetilde{M}_{\gamma}$ may not belong to~$\Lambda$, 
which leads to certain problems with applying, for instance, the disjointness arguments for the admissible iterations. 

In order to handle these problems,
for each connected component $I = (x_-,x_+)$ of some $\widetilde{M}_{\gamma}$ that intersects $\Lambda$, 
let us define $\hat{x}_+=\hat{x}_+^*=x_+$ if $x_+$ is a point of $\Lambda$ that is non-isolated on both sides. 
Otherwise, we take $\hat{x}_+$ and $\hat{x}_+^*$ to be respectively the left and the right endpoints of the 
connected component $J$ of $\mathbb{S}^1\setminus \Lambda$ such that $x_+\in\overline{J}$. 
Similarly, if $x_-$ belongs to such a $\overline{J}$, we take $\hat{x}_-$ and $\hat{x}_-^*$ 
to be respectively the right and left endpoints of $J$; otherwise, we 
let~$\hat{x}_-=\hat{x}_-^*=x_-$ (see Fig.~\ref{f:x-hat}). Finally, we denote 
$$
\hat{I}:=(\hat{x}_-,\hat{x}_+), \quad \hI_+:=(\hat{x}_-,x_+), \quad \hI_-:=(x_-,\hat{x}_+).
$$

\begin{figure}
\begin{center}
\includegraphics[scale=0.9]{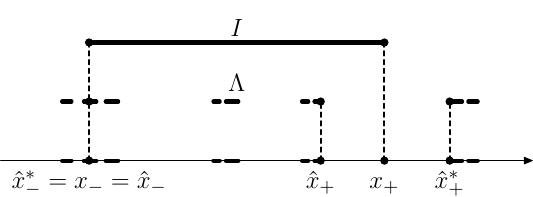} \includegraphics[scale=0.8]{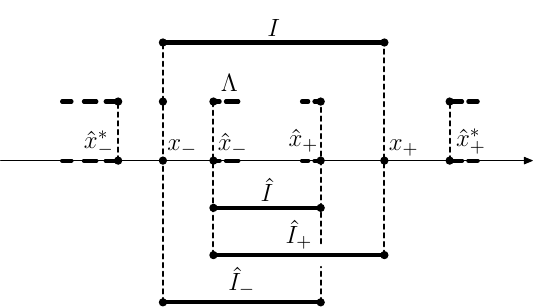} 
\end{center}
\caption{On the left: the interval $I$ (top), the minimal set $\Lambda$ (middle level), and two possible cases in the definition of 
$\hat{x}_{\pm}$, $\hat{x}_{\pm}^*$. On the right: the interval $I$, the set $\Lambda$, and the three 
intervals~$\hat{I}$, $\hat{I}_+$, $\hat{I}_-$.}\label{f:x-hat}
\end{figure}

Analogously to the minimal case, say that $g = \gamma_n \cdots \gamma_1$, 
$\gamma_1 \neq \gamma^{-1}$, is $\hI$-admissible if all intervals 
$\gamma_k \cdots \gamma_1(\hat{I})$, $k = 1,\ldots,n-1$, 
are disjoint from $\hat{I}$, and that it is a first-return if besides $g(\hat{I})$ intersects $\hat{I}$. 
The four lemmas below are analogs to Lemmas \ref{l:disj}, \ref{l:ends}, \ref{l:different}, \ref{l:lr-r}, and \ref{l:der=1}, 
respectively, that do hold in the minimal case.

\begin{lemma}
All the $\hI$-admissible images of $\hI_+$ are pairwise disjoint. The same holds for $\hI$-admissible images 
of $\hI_-$ (as well as of those of $\hI$ itself). Finally, each first-return~$g$ satisfies $g (\hat{I}) \subset \hat{I}$, 
$g (\hat{I}_{\pm}) \subset \hat{I}_{\pm}$,  and $\gamma_n = \gamma$. 
\end{lemma}

\begin{proof}
Due to the definition, the left endpoint $\hat{x}_-$ of $\hI_+$ is an accumulation 
point of~$\Lambda \cap \hI_+$. If two $\hI$-admissible images of $\hI_+$ under 
$g=\gamma_n\dots \gamma_1$ and $\overline{g}=\overline{\gamma}_m\dots \overline{\gamma}_1$ 
intersect, then due to the invariance of $\Lambda$, the interval 
$g(\hI_+)\cap \overline{g}(\hI_+)\subset \widetilde{M}_{\gamma_n} \cap \widetilde{M}_{\overline{\gamma}_m}$ 
intersects $\Lambda$ near its left endpoint. Hence, if both words $\gamma_n\dots \gamma_1$ and 
$\overline{\gamma}_m\dots \overline{\gamma}_1$ were nonempty, we would obtain 
$\gamma_n=\overline{\gamma}_m$ and get a contradiction in the same way as in 
the proof of Lemma~\ref{l:disj}.

Now, if $g$ is a first-return, the above arguments imply that $\gamma_n=\gamma$. 
Thus,~$g(\hI_+)~\subset~\widetilde{M}_{\gamma}$, and hence $g(\hI_+)\subset I$. 
Finally, the left endpoint of $g(\hI_+)$ is an accumulation point of~$\Lambda\cap g(\hI_+)$, 
which implies $g(\hI_+)\subset \hI_+$. In the same way we get $g(\hI_-)\subset \hI_-$, 
and putting together these two inclusions, we obtain $g(\hI)\subset \hI$.
\end{proof}

\begin{lemma}\label{l:mex-ends}
There exist first-returns $g_+$ and $g_-$, fixing~$x_+$ and~$x_-$, respectively. 
Moreover,
$$
g_+(\hat{x}_+)=\hat{x}_+, \quad g_+(\hat{x}_+^*)=\hat{x}_+^*, \quad 
g_-(\hat{x}_-)=\hat{x}_-, \quad g_-(\hat{x}_-^*)=\hat{x}_-^*.
$$
Finally, $g_- \neq g_+$. 
\end{lemma}
\begin{proof}
By repeating word by word the arguments of the proof of Lemma~\ref{l:ends}, we may find first-returns 
$g_{\pm}$ fixing the endpoints $x_{\pm}$ of~$I$, and we may check that $g_+\neq g_-$. 
Finally,~$g_+$ fixes $x_+$, and hence (due to the invariance of $\Lambda$) it also 
fixes~$\hat{x}_+$ and $\hat{x}_+^*$. The same arguments apply for $g_-$.
\end{proof}

\begin{lemma}\label{l:mex-lr-r}
If $g_+$ writes as $g_+ = \gamma_n \cdots \gamma_1$, then a right 
neighborhood of $\hat{x}_+^*$ is contained in $\widetilde{M}_{\widetilde{\gamma}}$, where 
$\widetilde{\gamma} := \gamma_n$ if $g_+$ is topologically contracting towards $\hat{x}_+^*$ 
in such a neighborhood, and $\widetilde{\gamma} := \gamma_1^{-1}$ otherwise. 
\end{lemma}
\begin{proof}
The same arguments as those of proof of Lemma~\ref{l:lr-r} (combined with 
Lemma~\ref{l:exterior-components} in case $\hat{x}_+^*\neq x_+$) prove this lemma.
\end{proof}

\begin{lemma}\label{l:mex-der=1} 
Neither $g_{-}$ nor $g_+$ have fixed points on $\hat{I}$. Moreover, 
$g_{-}'(x_{-})\ge 1$, and $g_{+}'(x_{+}) \ge 1$.
\end{lemma}

\begin{proof}
In the same way as in Lemma~\ref{l:der=1}, we notice that if $g_+$ had a fixed point inside~$\hI$, 
then in a small left neighborhood of $\hat{x}_+$ there would be a fundamental domain for its 
action intersecting~$\Lambda$. At any point of this domain (hence at some points of $\Lambda$), 
the sum of derivatives along all the geodesics would be uniformly bounded, which in contradiction 
with Lemma~\ref{l:inf-sum}. The inequality $g_{+}'(x_+)\ge 1$ follows using the same arguments 
as in Lemma~\ref{l:der=1}: otherwise, we would have $x_+\in \widetilde{M}_{\gamma}$.
\end{proof}

As we did in the minimal case, we denote by $\Phi_R(I)$ the connected component of 
$\widetilde{M}_{\gamma_1}$ intersecting $\gamma_1(I)$, where the element $g_+$ associated 
to $I$ writes as $g_+=\gamma_n\dots \gamma_1$. Again, we notice that $\Phi_R(I)$ 
and $\gamma_1(I)$ share the same right endpoint, as otherwise the 
point~$x_+~=~(\gamma_n\dots\gamma_2)(\gamma_1(x_+))$ would belong 
to $\widetilde{M}_{\gamma_n}=\widetilde{M}_{\gamma}$. Also, the map 
$\widetilde{g}_+$ corresponding to $\Phi_R(I)$ is 
$\gamma_1 g_+\gamma_1^{-1}=\gamma_1\gamma_n\dots \gamma_2$. 

Similarly, we define the map~$\JL$ associated to $g_-$. Finally, we denote by~$\hmI$ 
the set of connected components $I$ of sets $\widetilde{M}_{\gamma}$ that 
intersect~$\Lambda$. As in the minimal case, $\hmI$ decomposes into disjoint 
cycles corresponding to the application of~$\Phi_R$, and the same holds for~$\Phi_L$. 
Moreover, the next analog of Lemma \ref{l:finite} holds. 

\begin{lemma}\label{l:mex-finite}
There exists only a finite number of connected components of $\widetilde{M}_{\gamma}$ that intersect $\Lambda$.
\end{lemma}

\begin{proof}
For each $\hI$ corresponding to an interval $I \in \hmI$, the images $g_+(\hI)$ and $g_-(\hI)$ are 
disjoint. 
Hence, the derivative of at least one of the maps $g_+, g_-$ attains on~$\hI$ a value that does not 
exceed~$1/2$. On the other hand, by Lemma \ref{l:mex-der=1} , we have $g_-'(x_-) = g_+'(x_+) \ge 1$. 
Thus, we have a lower bound for the distortion of at least one of the maps $g_+$, $g_-$, and hence~$I$ 
belongs to either a left or a right cycle of intervals with sum of lengths bounded from below by 
$\frac{\log 2}{C_\mathcal{G}}$. Since there is only a finite number of such possible cycles, 
there are only finitely many connected components of each set~$\widetilde{M}_{\gamma}$ 
that intersect~$\Lambda$. 
\end{proof}

We have seen that the set $\hmI$ is finite. Moreover, Lemma~\ref{l:mex-lr-r} implies that immediately next 
to the right of a point $\hat{x}_+^*$ associated to an interval $I\in\hmI$, there is another interval of the same 
type. Hence,  except for a finite number of points, the set $\Lambda$ is covered by a finite family of intervals 
$\hI$. Denote this family by $\hmI_0$, and define $\wR : \bigcup_{\hI\in\hmI_0} \hI \to \mathbb{S}^1$ by 
$$\wR|_{\hI} = \gamma^{-1}|_{\hI}, \quad \forall \gamma \in \mathcal{G}, \quad \forall \hI\subset \widetilde{M}_{\gamma}.$$
Analogously to the minimal case, this induces a Markovian dynamics on $\Lambda$: the interval~$\gamma^{-1}(\hI)$ i
s a union of some intervals from $\hmI_0$ and connected components of $\mathbb{S}^1\setminus \Lambda$. Then 
we define the refined partitions~$\hmI_j$ as the set of connected components of all the preimages $\wR^{-1}(\bar{I})$, 
where $\bar{I}\in \hmI_{j-1}$. As in Lemma \ref{l:diameters}, we have 

\begin{lemma}\label{l:mex-diameters}
The diameter of the partition $\hmI_j$ tends to zero as $j\to\infty$.
\end{lemma}

\begin{proof}
Proceeding in a similar way to that of the minimal case, 
consider the first-exit map defined on each $\hI\in \hmI_0$ as
$$
\tR|_{\hI}(y)=\left\{\begin{array}{lll}
g_-^{-i}(y), &\quad & y\in (g_-^{i+2}(\hat{x}_+),g_-^{i+1}(\hat{x}_+)),\\
g_+^{-i}(y), &\quad & y\in (g_+^{i+1}(\hat{x}_{-}),g_+^{i+2}(\hat{x}_{-})),\\
y, &\quad & y\in (g_-^2(\hat{x}_+),g_+^2(\hat{x}_-)).
\end{array}\right.
$$
As before, we see that the distortion of the composition $\tR\circ \wR^j$ is bounded uniformly in 
$j$ on any interval of continuity. Thus, we have a uniform control for the distortion of the branches 
of the inverses of such compositions. The fact that the diameters of the refined partitions 
$\hmI_{j}$ tend to zero follows from this fact as in the minimal case. 
\end{proof}

Another application of the uniform control of distortion yields that the dynamics 
of the compositions $\tR\circ \wR^j$ becomes uniformly expanding for sufficiently 
large $j$. (Compare Corollary \ref{c:needed}.) As a consequence, the points in 
$\Lambda$ that are never sent by iterations of $\wR$ into points of the form 
$\hat{x}_-,\hat{x}_+$ are expanded under suitable group elements. Since the 
latter points are fixed by the elements $g_-,g_+$, this proves property~\sLmin, 
and hence closes the proof for free-group actions.

\vspace{0.25cm}

Let us now consider the general case. Recall that a result of Ghys states that every finitely-generated 
group of real-analytic circle diffeomorphisms having an exceptional minimal set is virtually free; see~\cite{ghys:topology}. 
Hence, let $H$ be a finite-index free subgroup of $G$, and let $g_1,\ldots,g_k$ be representatives of all left-classes of 
$H$ in $G$. We claim that $H$ preserves an exceptional minimal set (in particular, it has rank $\geq 2$, because 
of Denjoy's theorem). Otherwise, according to 
\cite{Ghys-actions,Navas-Es}, the group $H$ would either act minimally or admit 
a finite orbit~$A$. The former case is impossible, as $G$ does not act minimally. The latter case is also impossible, 
as it would imply that $G$ preserves the finite set $\bigcup_{i=1}^k g_i (A)$.

Let then $\widehat{\Lambda}$ be the exceptional minimal set for the action of $H$. We claim that the set 
$$\bigcup_{i=1}^{k} g_i (\widehat{\Lambda})$$
is the exceptional minimal set $\Lambda$ of $G$. Indeed, since $\Lambda$ is $H$-invariant, it contains a minimal 
$H$-invariant set, which must coincide with $\widehat{\Lambda}$ because of the uniqueness of exceptional minimal 
sets. As a consequence, the set $\bigcup_{i=1}^{k} g_i (\widehat{\Lambda})$ is $G$-invariant and contained in 
$\Lambda$. It must hence coincide with $\Lambda$, because of the minimality of $\Lambda$.

Since $H$ is a (necessarily higher-rank) free group, it must satisfy property 
$(\widehat{\Lambda} \star)$. Hence, due to Theorem~\ref{t:DKN; version exceptional}.3), 
the set of points of bounded $H$-expansivity in $\widehat{\Lambda}$ coincides with $H(\NE(H))$. As 
$\Lambda = \bigcup_{i=1}^{k} g_i (\widehat{\Lambda})$, and thus any point of $\Lambda$ can be mapped into $\widehat{\Lambda}$ by one of the $g_i$'s, the set of points of $\Lambda$ of bounded $G$-expansivity is a subset of 
$G(\NE(H))$. In particular, any non $G$-expandable point belongs to $G(\NE(H))$ and hence is fixed by a nontrivial 
element of $G$. We conclude that $G$ satisfies property~\sLmin, thus finishing the proof of the Main Theorem. 

\begin{remark} 
In the proof above, in cases where the set $\NE \cap \Lambda$ is nonempty, we have established property 
$\sLmin$ simultaneously with the Markovian property for the dynamics. (Strictly speaking, we have done 
this only for the case of the free group, but passing to finite extensions is straightforward.) This remark is 
not innoucous, as it shows that the approach of \cite{CC1,CC2,matsumoto} to the 
question of the Lebesgue measure of exceptional minimal sets, as well as the finiteness of the orbits 
of the connected components of the complement, pointed in the right direction. 
(See also \cite{inaba-matsumoto}.) Moreover, this again 
shows why it was worth considering two different cases in \S{}\ref{s:proof} instead of reducing the 
proof to a single one (see Remark~\ref{r:rare}). 
\end{remark}

We close this work with the

\vspace{0.15cm} 

\noindent{\bf Proof of Corollary \ref{c:finite-comp}.} Under the hypothesis, we have established property \sLmin. 
Let $y_1,\ldots,y_m$ be the non-expandable points (we allow the possibility $m=0$). For each index 
$i$, let $\widetilde{U}_i$ 
be a neighborhood of $y_i$ sufficiently small so that there exists~$g_i \in G$ fixing this point and having no other 
fixed point therein. By considering smaller neighborhoods $U_i \subset \widetilde{U}_i$, we may  assume that the 
first iterate of $g_i^{\pm 1}$ that sends a point $x \in U_i$ different from $y_i$ outside $\widetilde{U}_i$ has derivative 
$\geq 2$. Indeed, this follows from control of distortion just by writing $\widetilde{U}_i \setminus \{ y_i \}$ as the 
union of fundamental domains for the action of $g_i$ on the left and the right neighborhoods of $y_i$.

Points in the set $\Lambda \setminus \bigcup_i U_i$ are expandable, hence we may fix $\lambda > 0$ and 
a covering of this set by open intervals $V_1,\ldots,V_n$ such that on each $V_i$, certain element 
$h_i \in G$ expands by a factor at least $e^{\lambda}$. Now consider the open covering 
$$\Lambda \subset \bigcup_{i=1}^m U_i \cup \bigcup_{i=1}^n V_i.$$
There are only finitely many connected components of $\mathbb{S}^1 \setminus \Lambda$ that are not contained 
in one of these sets $U_i,V_i$; let us denote them by $I_1,\ldots,I_k$. We claim that the orbit of every connected 
component $I$ of $\mathbb{S}^1 \setminus \Lambda$ coincides with that of one of these intervals. Indeed, 
if $I$ does not coincide with none of them, then it is contained either in a certain interval $U_i$ or in 
some $V_i$. In the former case, the action of a suitable power of $g_i$ expands its length by a factor at least 
$2$, whereas in the latter, expansion by a factor at least $e^{\lambda}$ is performed by $h_i$. After this first 
expansion, either we have reached an interval of the form $I_j$, or the image remains in one of the $U_i,V_i$, 
so that we may repeat the procedure. Obviously, this expansion process must stop in finite time. At such a time, 
we have reached one of the intervals $I_j$, so that the orbit of this interval coincides with that of $I$. 
$\hfill\square$

\vspace{0.5cm}

\noindent{\bf{\Large Acknowledgments}}

The authors would like to thank Lawrence Conlon, \'Etienne Ghys, Steve Hurder, Yulij Ilyashenko, 
Shigenori Matsumoto, Julio Rebelo, and Dennis Sullivan, for many discussions all along this research 
project. We would also like to thank the anonymous referees for their careful work and helpful remarks.

B. Deroin's research was partially supported by ANR-08-JCJC-0130-01, ANR-09-BLAN-0116

V. Kleptsyn acknowledges the support of an RFBR project 13-01-00969-a and 
of a joint RFBR/CNRS project 10-01-93115-CNRS\_a

A. Navas acknowledges the support of the ACT-1103 project DySyRF.

Besides our host institutions, we thank the PUC-Chile, the CRM Barcelona, the 
Poncelet Laboratory in Moscow, the MRCC of Bedlewo, \emr{and the CIRM in Marseille,} 
for the very nice working conditions provided at several stages of this work.


\begin{small}

\vspace{0.1cm}

\noindent Bertrand Deroin\\
D\'epartement de Math\'ematiques et Applications
\'Ecole Normale Sup\'erieure\\
45 Rue d'Ulm, 75005 Paris\\
bertrand.deroin@ens.fr

\vspace{0.3cm}

\noindent Victor Kleptsyn\\
Institut de Recherches Math\'ematiques de Rennes (UMR 6625 CNRS)\\
Campus Beaulieu, 35042 Rennes, France\\
Victor.Kleptsyn@univ-rennes1.fr

\vspace{0.3cm}

\noindent Andr\'es Navas\\
Universidad de Santiago de Chile\\
Alameda 3363, Estaci\'on Central, Santiago, Chile\\
andres.navas@usach.cl

\end{small}

\end{document}